\documentclass[11pt,english]{amsart}

\usepackage{amsmath}
\usepackage{amsthm}
\usepackage{amssymb}
\usepackage{mathtools}
\usepackage[all]{xy}
\usepackage{todonotes}
\usepackage[pdfpagelabels,colorlinks,citecolor=magenta,pagebackref=true,urlcolor=magenta, bookmarksdepth=3,hyperindex,hyperfigures]{hyperref}
\hypersetup{pdfauthor   = {Karim Adiprasito and Geva Yashfe},
	pdftitle    = {The Partition Complex},
	}
\usepackage[nameinlink]{cleveref}
\usepackage{paralist}
\usepackage{nicefrac}
\usepackage{tikz-cd}
\usepackage{tikz}
\usepackage{bbm}
\usepackage[nodisplayskipstretch]{setspace}
\providecommand{\claimname}{Claim}
\providecommand{\definitionname}{Definition}
\providecommand{\lemmaname}{Lemma}
\providecommand{\propositionname}{Proposition}
\providecommand{\remarkname}{Remark}
\providecommand{\theoremname}{Theorem}
\providecommand{\corollaryname}{Corollary}
\providecommand{\examplename}{Example}
\providecommand{\motivationname}{Motivation}

\theoremstyle{plain}
\newtheorem{thm}{\protect\theoremname}
\theoremstyle{definition}
\newtheorem{defn}[thm]{\protect\definitionname}
\theoremstyle{plain}
\newtheorem*{lem*}{\protect\lemmaname}
\theoremstyle{remark}
\newtheorem*{mot}{\protect\motivationname}
\theoremstyle{plain}
\newtheorem{lem}[thm]{\protect\lemmaname}
\theoremstyle{plain}
\newtheorem{prop}[thm]{\protect\propositionname}
\theoremstyle{remark}
\newtheorem{rem}[thm]{\protect\remarkname}
\theoremstyle{remark}
\newtheorem*{claim*}{\protect\claimname}
\theoremstyle{plain}
\newtheorem{cor}[thm]{\protect\corollaryname}
\newtheorem{example}{\protect\examplename}

\newcommand{\K}{\mathbbm{k} }

\crefname{property}{property}{properties}
\Crefname{property}{Property}{Properties}
\Crefformat{property}{#2Property #1#3}

\begin{document}
\title{The Partition Complex: an invitation to combinatorial commutative algebra}
\author{Karim Adiprasito and Geva Yashfe}

\begin{abstract}
	We provide a new foundation for combinatorial commutative algebra and Stanley-Reisner theory using the partition complex introduced in \cite{AHL}. One of the main advantages is that it is entirely self-contained, using only a minimal knowledge of algebra and topology. On the other hand, we also develop new techniques and results using this approach.
	In particular, we provide
	\begin{compactenum}[1.]
		\item A novel, self-contained method of establishing Reisner's theorem and Schenzel's formula for Buchsbaum complexes.
		\item A simple new way to establish Poincar\'e duality for face rings of manifolds, in much greater generality and precision than previous treatments.
		\item A "master-theorem" to generalize several previous results concerning the Lefschetz theorem on subdivisions.
		\item Proof for a conjecture of K\"uhnel concerning triangulated manifolds with boundary.	
	\end{compactenum}
\end{abstract}
	
\maketitle

\section{Introduction}

Starting with the work of Hochster, Reisner and Stanley, powerful methods from commutative algebra developed by algebraic geometers could be used to provide a new and powerful way to study face numbers of simplicial and polyhedral complexes \cite{Hochster, Stanley96}.

However, using these powerful tools came with a drawback. First, they made the theory harder to access without background in commutative algebra. Second, even many of those applying them often used them as a black box, and the tools themselves became a distraction, leading to missed results and open questions that would otherwise have been simple.

And so, as a tourist might use an expensive lens to capture a vista, doing so suboptimally because he does not grasp its pros and cons, the physics of its makeup, we are left with pictures that feel somewhat lacking, blurry or hiding the important, leaving us dissatisfied. 

So our goal here is twofold: To show how basic household means can take a much simpler, more gratifying picture, without sacrificing any of the generality. We then go a step further, and use the new methods to generalize the results with ease, using only the ingredients that can be found within the first algebra books you can find in your kitchen, and just a smidge of algebraic topology you find in every spice rack. As for combinatorics, we shall assume nothing beyond the most basic familiarity with simplicial complexes.

Hence, this is not so much a survey, as it is an attempt to build better and more powerful foundations, as well as offer newcomers a road towards research in the area, that is at the heart of new developments between combinatorics and Hodge Theory \cite{MN, Karu, AHK, AHL}. Additionally, we offer also researchers in combinatorial commutative algebra a more consistent and stronger set of tools. We are therefore a little curt on direct combinatorial applications, for which we refer to the initial sections of \cite{AHL}, and instead offer an focused introduction to the techniques.

\subsection{Overview}

Before we begin discussing the details, let us provide a little motivation. We want to understand various combinatorial invariants of simplicial complexes. Most basic among these is the face vector, counting the number of vertices, edges, and so on. We may wish to restrict the class of complexes under investigation: for example, to look only at planar graphs, or at simplicial complexes that triangulate a surface. The restrictions we place are usually homological in nature.

The issue is then how the combinatorics and topology come together. The trick is to use rings which contain information from both worlds.

Indeed, one of the key observations of combinatorial commutative algebra was the realization that the homological properties of a simplicial complex are encoded in its so called \emph{face ring} in a variety of ways, often first glimpsed and disseminated as unpublished ideas and results of Hochster\footnote{This seems to justify the old adage that discoveries are never named after their discoverer, for the other name of face rings, Stanley-Reisner rings, makes no mention of Hochster}. The first key result here is Reisner's theorem (discussed in \Cref{sec:Reisner}), that connects the vanishing of homology over a fixed field to the Cohen-Macaulay property of the associated face ring. Here, not only the global homology of the simplicial complex comes into play, but also the homology of principal filters in the face lattice.

The essentially only proof available for this theorem goes via the local cohomology as introduced by Grothendieck in the 1960s, and most of the following research has similarly employed the same tool. We instead use the partition complex here, a significantly more down-to-earth tool that has several direct benefits, most of all that one can see what happens in a surgical way.

We also obtain the generalization to manifolds, due to Schenzel \cite{Schenzel81}, which is relatively transparent at least to experts, but has the drawback that it is, in parts, only available in his German thesis. We provide this in \Cref{sec:Schenzel}.

Our next stop in the way is a new way to address and understand a fundamental property of intersection rings that arises in the context of combinatorial Hodge theory: Poincar\'e duality. Again we offer a new transparent proof of Poincar\'e duality for the face rings of spheres, and then proceed to provide generalizations to arbitrary manifolds, discussed in \Cref{sec:Poincare}.

Finally, we discuss some new applications to face number problems for manifolds. In \Cref{sec:subdivisions}, we discuss the connection to subdivisions and Lefschetz properties, and provide a far-reaching subdivision theorem, providing a common generalization of previous works in one swoop. We also discuss related conjectures of K\"uhnel, concerning small triangulations of manifolds.

\textbf{Acknowledgements} We would like to thank Zuzka Pat\'akov\'a and Hailun Zheng for an attentive reading of our paper, helping us correct many typos and provide useful remarks to improve understanding. We also thank the anonymous referee for useful remarks.
Karim Adiprasito is supported by the European Research Council under the European Unions Seventh Framework Programme ERC Grant agreement ERC StG 716424 - CASe, a DFF grant 0135-00258B and the Israel Science Foundation under ISF Grant 1050/16.
Geva Yashfe is supported by the European Research Council under the European Unions Seventh Framework Programme ERC Grant agreement ERC StG 716424 - CASe and the Israel Science Foundation under ISF Grant 1050/16.

\section{Preliminaries}
In this section we set up some basic notation and definitions. Experienced readers can skip most of the text, but may still wish to look at the notation and at definitions for relative simplicial complexes, as well as the corresponding modules over face rings.

\subsection{Simplicial complexes and face rings}
\subsubsection{Simplicial complexes} We begin by recalling some common definitions.
\begin{defn}
	A simplicial complex $\Delta$ is a downwards-closed family of subsets of a finite set called the ground set. The ground set is usually left implicit or taken to be $[n]=\{1,\ldots,n\}$ for some $n$. Being downwards-closed means that if $\tau\in\Delta$ and $\rho\subset\tau$ then $\rho\in\Delta$.
	
	In particular, if a simplicial complex is nonempty, it contains $\emptyset$ as a face. Thus the complex $\{\emptyset\}$ contains no nonempty faces, but is different than the void complex $\emptyset$.
	
	A subcomplex of a simplicial complex is a subset which is itself a simplicial complex.
	
	An element $\tau\in\Delta$ is called a simplex or a face. Its dimension is $\dim(\tau)=|\tau|-1$, and the dimension of $\Delta$ is $\max_{\tau\in\Delta}\dim(\tau)$. A face of $\Delta$ is called a facet if its dimension equals $\dim(\Delta)$. Faces of dimension zero and one are called vertices and edges respectively.
\end{defn}

\begin{defn}
	Let $\Delta$ be a simplicial complex. The \emph{star} of a simplex $\tau$ is the simplicial complex $\mathrm{st}_\tau(\Delta) = \{\rho\in\Delta \mid \tau\cup\rho\in\Delta\}$. The \emph{link} of $\tau$ is $\mathrm{lk}_\tau(\Delta)=\{\rho\in\Delta \mid \tau\cup\rho\in\Delta, \tau\cap\rho=\emptyset\}$.
	
	The $k$-faces of $\Delta$ are denoted by $\Delta^{(k)} = \{\tau\in\Delta \mid \dim(\tau)=k\}$, and the $k$-skeleton $\Delta^{(\le k)}$ is the subcomplex consisting of faces of dimension at most $k$.
\end{defn}

In one or two places we use the join and subtraction operations. For simplicial complexes $\Delta_1,\Delta_2$ on disjoint ground sets, the join is $\Delta_1^\ast\Delta_2=\{\tau\cup\rho \mid \tau\in\Delta_1, \rho\in\Delta_2\}$, a simplicial complex on the union of the ground sets of $\Delta_1$ and $\Delta_2$.

If $\Delta$ is a simplicial complex and $\tau$ is a face, $\Delta-\tau$ is the maximal subcomplex which does not contain $\tau$. Its faces are $\{\sigma\in\Delta\mid\sigma\cap\tau=\emptyset\}$.

It is worth noting that simplicial complexes are not equivalent to semi-simplicial sets (sometimes called $\Delta$-complexes by Hatcher). 

\subsubsection{Relative simplicial complexes}
We work with relative simplicial complexes analogously to how one often works with pairs of topological spaces. The theory generalizes smoothly to this setting, which is sometimes cleaner. See also \cite{Adiprasitorelative,AHL}.

\begin{defn}
	A relative simplicial complex $\Psi=(\Delta,\Gamma)$ is a pair consisting of a simplicial complex $\Delta$ and a subcomplex $\Gamma$. Its faces are $\Delta\setminus\Gamma$, i.e. the non-faces of $\Gamma$. In particular, $\dim(\Psi)=\max_{\tau\in\Psi}\dim(\tau)$ can be smaller than $\dim(\Delta)$, and it is possible for $\emptyset$ not to be a face. Any simplicial complex $\Delta$ can be treated in this language as the relative complex $(\Delta, \emptyset)$.
	
	The star of a simplex $\tau$ within $\Psi$ is $\mathrm{st}_\tau \Psi = (\mathrm{st}_\tau \Delta, \mathrm{st}_\tau \Gamma)$. Similarly, the link is $\mathrm{lk}_\tau \Psi = (\mathrm{lk}_\tau \Delta, \mathrm{lk}_\tau \Gamma)$.
	
	A relative complex $\Psi$ is pure if all its maximal faces have the same dimension.
\end{defn}

Many basic lemmas about simplicial complexes work for relative complexes as well, and we will often extend definitions from absolute to relative without further mention. For instance, if $\Psi=(\Delta,\Gamma)$ and $\tau\in\Delta$ then $\mathrm{st}_\tau \Psi=\tau^\ast\mathrm{lk}_\tau \Psi=(\tau^\ast\mathrm{lk}_\tau\Delta, \tau^\ast\mathrm{lk}_\tau\Gamma)$. Note that the join of any complex with the void complex is void.

The open star of a face $\tau$ in a simplicial complex is usually defined to be the set of faces containing $\tau$. This is not a subcomplex in the usual sense, but we can define a relative complex to fill the same role.

\begin{defn}
	Let $\Delta$ be a simplicial complex. The open star of a face $\tau$ is $\mathrm{st}_\tau^\circ\Delta = (\mathrm{st}_\tau\Delta, \mathrm{st}_\tau\Delta - \tau).$
\end{defn}

\subsubsection{Homology of complexes}
The cohomology $H^\ast(\Psi;\K )$ of a relative complex $\Psi=(\Delta,\Gamma)$ is the simplicial cohomology of the pair with coefficients in $\K$. For a complex $\Delta$, we consider $\emptyset\in\Delta$ as a face (of dimension $-1$) for this purpose. Thus our $H^\ast(\Delta)$ is what is often denoted $\tilde{H}^\ast(\Delta)$. In particular the void complex $\emptyset$ has vanishing cohomology in all dimensions, but $\Delta=\{\emptyset\}$ has 
\[H^i(\Delta;\K )=
\begin{cases}
\K & i=-1 \\
0 & \mathrm{otherwise.}
\end{cases}
\]

\subsubsection{Face rings}
Face rings, or Stanley-Reisner rings, are main object of the paper. Our treatment is standard except for the relative case, in which we follow \cite{Adiprasitorelative} and \cite{AHL}.

Fix a field $\K $. Except in \Cref{sec:CM}, $\K $ is assumed to be infinite. This is a harmless assumption, as field extensions change no property that interests us in this context.

\begin{defn}
	Let $\Delta$ be a simplicial complex. Define the polynomial ring $\K [x_v \mid v\in \Delta^{(0)}]$, with variables indexed by vertices of $\Delta$. The Stanley-Reisner ideal (or non-face ideal) $I_\Delta$ of $\Delta$ is the ideal generated by all elements of the form $x_{v_1}\cdot x_{v_2} \cdot\ldots\cdot x_{v_j}$ where $\{v_1,\ldots,v_j\}$ is not a face of $\Delta$.
	
	The Stanley-Reisner ring (or face ring) of $\Delta$ is
	\[\K [\Delta] \coloneqq \K [x_v \mid v\in \Delta^{(0)}] / I_\Delta.\]
	
	If $\Psi=(\Delta,\Gamma)$ is a relative complex, the \emph{relative face module} of $\Psi$ is defined by $I_\Gamma / I_\Delta$. This is an ideal of $\K [\Delta]$.
\end{defn}

Two main types of maps between face rings and modules are used in this paper. If $\Psi=(\Delta,\Gamma)$ and $\Psi'=(\Delta,\Gamma')$ are relative complexes such that $\Gamma' \subset \Gamma$, there is an inclusion map
\[\K[\Psi] \hookrightarrow \K[\Psi'].\]
Similarly, if $\Psi=(\Delta,\Gamma)$ and $\Psi'=(\Delta',\Gamma)$ such that $\Delta' \subset \Delta$ is a subcomplex, there is a restriction map
\[\K[\Psi] \twoheadrightarrow \K[\Psi'].\]
In general, maps do not exist in the opposite direction. Two particularly relevant examples are the inclusion of an open star into a complex and the restriction to the star of a face. Explicitly, for $\Psi=(\Delta,\Gamma)$ and any $\tau \in \Delta$, these are maps
\[\K[\mathrm{st}^\circ_\tau \Psi] \simeq \K[\Delta,\mathrm{st}_\tau\Gamma\cup(\Delta-\tau)] \rightarrow \K[\Psi] \]
and
\[\K[\Psi] \rightarrow \K[\mathrm{st}_\tau \Psi]\]
respectively.

\subsubsection{Gradings of face rings}
Face rings can be \emph{graded} by monomial degree. That is, if $\Delta$ is a complex, we can write
\[\K[\Delta] = \bigoplus_{n\ge 0}\K[\Delta]_n,\]
where the direct sum is a sum of vector spaces over $\K$, and $\K[\Delta]_n$ is the subspace spanned by monomials of degree $n$. This is called the coarse grading. An element of $\K[\Delta]$ is homogeneous if it is in a single graded piece, or in other words, if it is a linear combination of monomials having the same degree.

There is also a finer grading, by the exponent vectors of monomials. If $\Delta$ has vertices $\{v_1,\ldots,v_k\}$ and $x^\alpha = x_{v_1}^{\alpha_1} \cdot \ldots\cdot x_{v_k}^{\alpha_k}$ is a monomial, its exponent vector is $(\alpha_1,\ldots,\alpha_k) \in \mathbb{Z}_{\ge 0}^{\Delta^{(0)}}$. The piece of $\K[\Delta]$ in degree $\alpha$ is the span of this monomial. This is the fine grading.

Note that in both cases, the degree of a product of homogeneous elements is the sum of their degrees.

Given a graded module or algebra, one can encode the dimensions of the graded pieces in a generating function. This is called the \emph{Hilbert series}, and we shall focus mostly on the Hilbert series of a face ring with respect to the coarse grading
\[H(\K[\Delta])(t)\  = \sum_{i=0}^\infty \dim_{\K }((\K_i[\Delta]))\cdot t^i.
\]
\begin{mot}
This interests us because the Hilbert series of a simplicial complex is also combinatorial:
\[H(\K[\Delta])(t)\ =\ \frac{1}{(1-t)^n}\sum_{i=0}^d f_{i-1} t^i(1-t)^{n-i}.\]
\end{mot}

The same discussion applies verbatim to relative face modules. The relevance of this is that maps between modules often preserve the degree. In this case, we can often understand a complex of maps most easily by examining each graded piece separately.

\subsection{Chain complexes}
We discuss some definitions and basic lemmas for chain and double complexes, and provide a basic introduction. If you have not seen chain complexes before, we recommend Hatcher for a basic introduction \cite{Hatcher}.

\begin{defn}[Chain complexes and tensor products]
	All our complexes are cohomologically graded. That is, our chain complexes are denoted $C^\ast$, with differential $C^\ast\rightarrow C^{\ast+1}$. It is convenient to call $H^i(C^\ast)$ the $i$-th homology, rather than cohomology, of $C^\ast$. 
	
	To shift the index by $p$, we write $C^{\ast+p}$ (and $(C^{\ast+p})^i=C^{i+p}$).
	
	If $(B^\ast,d)$ and $(C^\ast,d')$ are chain complexes, their tensor product is the double complex $T^{\ast,\ast}$ defined by 
	\[T^{i,j} = B^i \otimes C^j \]
	together with maps $d^h = d\otimes\mathrm{id}:T^{i,j}\rightarrow T^{i+1,j}$ and $d^v = \mathrm{id}\otimes d':T^{i,j}\rightarrow T^{i,j+1}$. If $B^\ast,C^\ast$ are complexes of modules over some ring $R$, the tensor product is of $R$ modules, i.e. it is $B^i\otimes_R C^j$.
	Note the convention here is that the squares of the complex commute. 
	
	A small piece of $T^{\ast,\ast}$ can be pictured as follows.
	
	\[
	\xymatrix{ & \ldots & \ldots\\
		\ldots\ar[r]^{d^{h}} & B^{i}\otimes C^{j+1}\ar[r]^{d^{h}}\ar[u]^{d^{v}} & B^{i+1}\otimes C^{j+1}\ar[r]^{\quad d^{h}}\ar[u]^{d^{v}} & \ldots\\
		\ldots\ar[r]^{d^{h}} & B^{i}\otimes C^{j}\ar[r]^{d^{h}}\ar[u]^{d^{v}} & B^{i+1}\otimes C^{j}\ar[r]^{\quad d^{h}}\ar[u]^{d^{v}} & \ldots\\
		& \ldots\ar[u]^{d^{v}} & \ldots\ar[u]^{d^{v}} & \ldots
	}
	\]
\end{defn}
\subsection{Double complexes}
The main proofs of the paper are established using the homology of double complexes, the homological way to perform what combinatorialists know well as double counting. To do this in a manner as accessible as possible, without leaving too much for the reader, we use mapping cones very extensively.

 Everything we need is introduced below.

We begin with some notation.
\begin{defn}
	Let $C^{\ast,\ast}$ be a double complex with commuting differentials $d^h$ and $d^v$ (our double complexes always have commuting differentials). Each row $C^{\ast,j}$ and each column $C^{i,\ast}$ is a chain complex with differential induced from $d^h$ or $d^v$ respectively.
	
	The \emph{total complex} of $C$ is a chain complex given by $\mathrm{Tot}(C)^k = \bigoplus_{i+j=k}C^{i,j}$ and differential $d_\mathrm{Tot}^k : \mathrm{Tot}(C)^k \rightarrow \mathrm{Tot}(C)^{k+1}$ defined by either $d^h + (-1)^k d^v$ or $d^v + (-1)^k d^h$. These give equivalent homology, and it is convenient to have both (an alternative is to transpose the complex, but both are used for the same double complex here).
	
	We denote elements $\alpha \in \mathrm{Tot}(C)^k$ by sums $\alpha = \sum_{i+j=k}\alpha^{i,j}$, where it is understood that $\alpha^{i,j}\in C^{i,j}$.
	
	The \emph{truncation} $C^{\ast\ge i_0, \ast}$ is a double complex defined by
	\[ 
	(C^{\ast\ge i_0, \ast})^{i,j} = C^{i\ge i_0, j} = 
	\begin{cases}
		C^{i,j} & i\ge i_0 \\
		0 		& \mathrm{otherwise,}
	\end{cases}
	 \]
	 with the same differentials as $C^{\ast,\ast}$, and $0$ for $i<i_0$. 
	 
	 The truncation $C^{\ast,\ast\le j_1}$ is defined analogously.
\end{defn}

Our goal for the rest of this section is to produce exact sequences tying together the rows, columns, and total complex of a double complex. We do this using mapping cones. The idea is introduced after a little preparation.

\begin{lem}\label[lemma]{lem:exactness}
	Let $C^{\ast,\ast}$ be a bounded double complex. For $H^k(\mathrm{Tot}(C))$ to vanish, it suffices that the homology in the vertical direction of $C^{i,k-i}$ is zero for each $i$, i.e. that $H^{k-i}(C^{i,\ast})=0$ for all $i$. Similarly, it suffices that $H^{k-i}(C^{\ast,i})=0$ for all $i$.
\end{lem}
\begin{proof}
	We show this for the vertical case, the horizontal one being analogous.
	Let $\sum_{i+j=k}\alpha^{i,j} \in \mathrm{Tot}(C)^k$ be a cycle, and let $i_0$ be the minimial index such that $\alpha^{i_0,k-i_0}\neq 0$. Then $d^v(\alpha^{i_0,k-i_0})=0$, so by assumption there is some $\beta=\beta^{i_0,k-i_0-1}$ mapping to $\alpha^{i_0,k_0}$ under $d^v$.
	\[\xymatrix{\ddots\ar[r] & \hole\\
		& 0\ar[r]\ar[u] & 0\\
		&  & \alpha^{i_{0},k-i_{0}}\ar[r]\ar[u] & 0\\
		&  & \beta\ar[r]\ar[u] & \alpha^{i_{0}+1,k-i_{0}-1}\ar[r]\ar[u] & 0\\
		&  &  &  & \ddots\ar[u]
	}
	\]
	Thus $\alpha' = \alpha - ((-1)^{k-1}d^h +d^v)(\beta)$ differs from $\alpha$ by a boundary. Replacing $\alpha$ by $\alpha'$ increases the minimal nonvanishing index $i_0$, and after finitely many steps the process terminates because $C^{\ast,\ast}$ is bounded.
\end{proof}

\begin{cor}
	If all rows or all columns of a double complex are exact then the total complex is acyclic.
\end{cor}

We introduce maps $\mathfrak{R,U}$ (for ``right" and ``up") between columns (respectively rows) of a double complex and the total complexes of certain truncations.

\begin{defn}
	Let $C^{\ast,\ast}$ be a double complex. There is a chain map
	\[\mathfrak{R}^i:C^{i,\ast} \rightarrow \mathrm{Tot}(C^{\ast\ge i+1, \ast})^{\ast+i+1},\]
	from the $i$-th column to the total complex of a truncation of $C^{\ast,\ast}$,
	which is given by
	\[C^{i,j} \rightarrow \mathrm{Tot}(C^{\ast\ge i+1, \ast})^{i+j+1}\]
	\[\alpha \mapsto d^h(\alpha)\in C^{i+1,j} \subset \bigoplus_{\substack{r+s=i+j+1, \\ r\ge i+1}}C^{r,s}.\]
	For the signs make $\mathfrak{R}$ commute with the differentials, the differential of the total complex is taken to be $d^v + (-1)^k d^h$. 
	
	This is illustrated below, with summands of $\mathrm{Tot}(C^{\ast\ge i+1, \ast})^{i+j+1}$ underlined.
	\[\xymatrix{\hole\\
		C^{i,j+1}\ar[u] & \hole\\
		C^{i,j}\ar[u]\ar[r] & {\underline{C^{i+1,j}}}\ar[r]\ar[u] & \hole\\
		C^{i,j-1}\ar[u] &  & {\underline{C^{i+2,j-1}}}\ar[r]\ar[u] & \hole\\
		\vdots &  &  & {\color{red}{\color{blue}{\ddots}}}
	}
	\]
	
	There is a similar chain map
	\[\mathfrak{U}^j:\mathrm{Tot}(C^{\ast,\ast\le j})^\ast\rightarrow C^{\ast-j,j+1},\]
	from the total complex of a truncation of $C^{\ast,\ast}$ to the $j+1$-th row. On an element
	\[\alpha = \sum_{r+s=k}\alpha^{r,s}\in \mathrm{Tot}(C^{\ast,\ast\le j})^k\]
	we define it by 
	\[\mathfrak{U}^j(\alpha) = d^v(\alpha^{k-j,j}).\] 
	Here the differential of the total complex should be taken to be $d^h +(-1)^k d^v$.
\end{defn}

\begin{defn}[Mapping cones]
	Let $f:(C^\ast,\partial)\rightarrow (C'^\ast,\partial')$ be a map of chain complexes. The mapping cone of $f$ is the chain complex $(M(f)^\ast,d)$, where $M(f)^i = C^i\oplus C'^{i-1}$ and $d^i(\alpha,\beta)=(\partial\alpha, \partial'\beta + (-1)^i f\alpha)$.
	
	Given $f$, we can construct a map of chain complexes $\iota:C'^{\ast-1}\rightarrow M(f)^\ast$ by $\beta\mapsto (0,\beta)$. This fits into a short exact sequence
	\[0 \rightarrow C'^{\ast-1} \overset{\iota}\rightarrow M(f)^\ast \rightarrow C^\ast \rightarrow 0,\]
	which gives rise to a long exact sequence
	\[\ldots\rightarrow H^{i}(C)\rightarrow H^{i}(C')\rightarrow H^{i+1}(M(f))\rightarrow H^{i+1}(C)\rightarrow\ldots\]
	in which the connecting homomorphism is induced by $f$.
\end{defn}

The next lemma is the essential point.
\begin{lem}\label[lemma]{Tot_mapping_cone}
	Let $C^{\ast,\ast}$ be a double complex with commuting vertical and horizontal maps $d^{h},d^{v}$. There are isomorphisms
	\begin{align*}
		M(\mathfrak{R}^i) &\simeq \mathrm{Tot}(C^{\ast\ge i,\ast})^{\ast+i} \\
		M(\mathfrak{U}^j) &\simeq \mathrm{Tot}(C^{\ast,\ast\le j+1})^\ast.
	\end{align*}
\end{lem}
\begin{proof}
	First consider $f=\mathfrak{R}^i : C^{i,\ast} \rightarrow \mathrm{Tot}(C^{\ast\ge i+1, \ast})^{\ast+i+1}$. By definition,
	\[M(f)^j = C^{i,j} \oplus \mathrm{Tot}(C^{\ast\ge i+1, \ast})^{i+j+1-1} = \bigoplus_{\substack{r+s=i+j \\ r \ge i}}C^{r,s} = \mathrm{Tot}(C^{\ast\ge i, \ast})^{i+j}, \]
	and the differential of $M(f)^\ast$ is essentially the same as that of $\mathrm{Tot}(C^{\ast\ge i, \ast})$.
	
	Now consider $f=\mathfrak{U}^j:\mathrm{Tot}(C^{\ast,\ast\le j})^\ast\rightarrow C^{\ast-j,j+1}$. This time
	\[M(f)^i = \mathrm{Tot}(C^{\ast,\ast\le j})^i \oplus C^{i-1-j,j+1} = \bigoplus_{\substack{r+s = i \\ s \le j+1}}C^{r,s} = \mathrm{Tot}(C^{\ast,\ast\le j+1})^i,\]
	and the differential of $M(f)^\ast$ is the same as that of the total complex on the right hand side above if $j$ is even. If $j$ is odd, it is harmless to modify the differential of the mapping cone to be
	\[d^i(\alpha,\beta)=(\partial\alpha, \partial'\beta + (-1)^{i-1} f\alpha)\]
	instead of the expression above: the two expressions give isomorphic complexes $M(f)^\ast$.
\end{proof}

\begin{rem}
	Mapping cones are a construction in homological algebra, motivated by a similar construction in algebraic topology. They are found in most textbooks on homological algebra, sometimes with slightly different indexing or sign conventions. The topological construction from which they originate is described, for instance, in chapter $0$ of Hatcher's text \cite{Hatcher}.
\end{rem}

\section{Cohen-Macaulay Complexes and why we care}\label[section]{sec:CM} 

Let us now turn to the little bit of commutative algebra necessary for our purposes. We refer to \cite{AM} for a general account, and \cite{Bruns-Herzog} for something a little more specialized to our situation.

\subsection{The Basic Idea}\label[section]{sec:basic_idea}
Consider a simplicial complex $\Delta$ and its face ring $\K [\Delta]$: if $\Delta$ has at least one vertex $v$, this is a graded ring with $\K [\Delta]_i\neq 0$ for each $i$. Indeed $x_v^i \in \K [\Delta]_i$. Thus, as a vector space over $\K $, each graded piece has finite dimension, but the entire ring is always infinite dimensional. It is useful to work with a finite-dimensional $\K $-algebra instead, provided it preserves enough information about $\K [\Delta]$. The idea then is to "peel" $\K [\Delta]$ by quotienting out an ideal which is as large as reasonably possible. That $\K [\Delta]$ is Cohen-Macaulay means this peeling can be performed especially nicely, as we shall soon see.

The importance of all this is due to the fact that $\K [\Delta]$ is always Cohen-Macaulay if $\Delta$ is the link or star of a face in any triangulation of a manifold with boundary (and in particular if $\Delta$ is a disk or a sphere). This is a shadow of the fact that each point of a manifold has a neighborhood with trivial topology.

We put rings and modules on an equal footing, so these tools are later available for relative face modules.

\begin{defn}
	Let $R$ be a ring and $M$ an $R$-module. A \emph{regular sequence} on $M$, or \emph{$M$-sequence}, is a sequence of elements $(\theta_1,\ldots,\theta_n)$ in $R$ such that:
	\begin{enumerate}
		\item Each $\theta_i$ is a nonzerodivisor on $M / \langle\theta_1,\ldots,\theta_{i-1}\rangle$, and
		\item $M/\langle\theta_1,\ldots,\theta_n\rangle\neq 0$.
	\end{enumerate}
	If $R$ is a graded ring, a sequence as above is called homogeneous if each $\theta_i$ is.
\end{defn} 

We care mainly about homogeneous regular sequences, and among them mainly about those in which all elements have degree $1$. We will see that if the field $\K$ is infinite, $A$ and $M$ are graded with $A$ generated in degree $1$ (in particular $A_0 = \K$), and there exists an $M$-sequence of length $n$, then there also exists an $M$-sequence consisting of degree-$1$ elements.

Quotienting a graded $\K $-algebra $A$ by a regular sequence of degree-$1$ elements is an operation which is well-behaved with respect to the \emph{Hilbert series} of $A$, which we previously encountered in the case $A=\K [\Delta]$. For a graded $\K $-vector space $V$, the Hilbert series is 
\[H_V(t) = \sum_{i=0}^\infty \dim_{\K }(V_i)\cdot t^i.\]
Consider the ideal $\langle\theta_1\rangle$ generated by a nonzerodivisor of $A$ having degree $1$. Since the multiplication map by $\theta_1$ is an injection of vector spaces $A\to A$ which increases the degree by $1,$ we have
\[\dim_{\K }(\langle\theta_1\rangle_i) = \dim_{\K }(A_{i-1}),\]
so
\[H_{\langle\theta_1\rangle}(t) = \sum_{i=0}^\infty \dim_{\K }(A_{i-1})\cdot t^{i} = t\cdot H_{A(t).} \]
In particular, we find that
\[H_{A/\langle\theta_1\rangle}(t) = H_{A}(t) - H_{\langle\theta_1\rangle}(t) = (1-t)\cdot H_{A}(t).\]
Modding out by the ideal $\langle\Theta\rangle$ generated by a regular sequence $\Theta=(\theta_1,\ldots,\theta_n)$ consisting of degree-$1$ elements therefore gives
\[H_{A/\langle\Theta\rangle}(t) = (1-t)^n \cdot H_{A}(t).\]
All this works in just the same way if we instead work with a regular sequence of degree-$1$ elements on an $A$-module $M$.

\begin{defn}
	Let $A$ be a finitely-generated graded $\K $-algebra and let $M$ be a finitely-generated $A$-module. A \emph{homogeneous system of parameters} (h.s.o.p.) for $M$ is a sequence of homogeneous elements $\Theta=(\theta_1,\ldots,\theta_n)$ of $A$ of minimal length among those sequences satisfying that the quotient $M/\langle\Theta\rangle$ is finite dimensional over $\K $. 
\end{defn}

The length of a h.s.o.p. as above is the Krull dimension of the support of $M$ (while this fact gives important context, we will have no further use for it). For a face ring $\K [\Delta],$ it is always $\dim(\Delta)+1$, or equivalently the maximum cardinality of a face. For a relative face module $\K [\Delta,\Gamma]$, it is the maximum cardinality of a face of $\Delta$ which is not contained in $\Gamma$.

\begin{defn}
	Let $A$ be a finitely-generated graded $\K $-algebra, $M$ a finitely-generated graded $A$-module. Then $M$ is \emph{Cohen-Macaulay} if it has a homogeneous system of parameters which is an $M$-sequence.
\end{defn}

It is not difficult to show that if $M$ is Cohen-Macaulay and has h.s.o.p. of length $n$, then any $M$-sequence of length $n$ is also an h.s.o.p., and no longer $M$-sequence can exist. Therefore, again assuming $k$ is infinite and $A$ is generated in degree $1$, this sequence may be chosen to consist of degree $1$ elements of $A$.

Hence, under all these assumptions, a Cohen-Macaulay $M$ can be "peeled" as nicely as can be hoped: there is an $M$-sequence $\Theta=(\theta_1,\ldots,\theta_n)$ such that
\[(1-t)^n \cdot H_M(t) = H_{M/\langle\Theta\rangle}(t)\]
is a polynomial. That the left-hand side is a polynomial is true even if the Cohen-Macaulay assumption is omitted. It is the equality to the Hilbert series on the right that is exceptional. A numerical consequence is that the coefficients of this polynomial are positive. More important for us is that the dimensions of graded pieces of $M/\langle\Theta\rangle$ are related to those of $M$ by an explicit formula depending on $n$ alone, and that $M/\langle\Theta\rangle$ is finite dimensional.

\subsubsection{An explicit calculation} Let us see what these dimensions are. First, the Hilbert series $H_{\K [\Delta]}$ is controlled by the $f$-vector of $\Delta$ in the following way. Each non-vanishing monomial in $\K [\Delta]$ is supported on a unique face of $\Delta$, and the set of monomials with given support $\{v_{i_1},\ldots,v_{i_m}\}$ is
\[\left\{\prod_{j=1}^m x_{i_j}^{s_j} \mid 1 \le s_1,\ldots,s_m\right\}.\]
The span of this set is a graded $\K $-vector space with Hilbert series equal to the rational function 
\[\frac{x^m}{(1-x)^m} = \frac{x^m (1-x)^{d+1 - m}}{(1-x)^{d+1}}\]
where $d = \dim(\Delta)$. Summing over faces (and keeping in mind the difference between face dimension and cardinality) we obtain
\[H_{\K [\Delta]} = \frac{\sum_{i=-1}^{d} f_i \cdot x^{i+1} (1-x)^{d - i}}{(1-x)^{d+1}}. \]
Thus if $\K [\Delta]$ is Cohen-Macaulay, its quotient by a regular sequence of length $d+1$ has Hilbert series (now polynomial) $\sum_{i=-1}^d f_i \cdot x^{i+1} (1-x)^{d - i},$ and one can verify it equals $\sum_{i=0}^{d+1} h_i x^i$ where $(h_0,\ldots,h_{d+1})$ is the $h$-vector of $\Delta$. The entries of the $h$-vector have tremendous meaning, and tell us more intimately what a simplicial complex is about than the face numbers or the Betti numbers do. For instance, they display some fundamental symmetries of the face vector. When we prove the Poincar\'e duality theorem for face rings (\cref{thm:pd}), this will imply the famous Dehn-Sommerville relations \cite{Sommerville}: for spheres $\Delta$ of dimension $d-1$, we have the fundamental symmetry
\[h_i\ =\ h_{d-i}.\]
Think of this as a generalization of the fact that the alternating sum of the face numbers equals the Euler characteristic. This fact can be seen as a special case: it is equivalent to the relation $h_0\ =\ h_{d}$.

\subsubsection{Associated primes and prime avoidance}
We need the following notions from commutative algebra.

Let $A$ be a finitely generated $\K$-algebra and let $M$ be an $A$-module.

\begin{defn}
	A prime ideal $\mathfrak{p}\subset A$ is \emph{associated to} $M$ if there exists an $m\in M$ such that
	\[\{a\in A\mid am = 0\} = \mathfrak{p},\]
	i.e. if $\mathfrak{p}$ is the annihilator of $m$.
\end{defn}
\begin{lem}\label[lemma]{lem:associated}
	The set of zerodivisors on $M$ equals the union of the primes associated to $M$. That is,
	\[\{a\in A\mid am = 0 \text{ for some $0\neq m\in M$}\} = \bigcup_{\substack{\mathfrak{p} \subset A \text{ prime}\\ \text{$\mathfrak{p}$ is associated to $M$}}} \mathfrak{p}.\]
\end{lem}
\begin{lem}\label[lemma]{lem:prime_avoidance}
	Let $I \subset A$ be an ideal and let $\mathfrak{p}_1,\ldots\mathfrak{p}_n \subset A$ be prime ideals. If
	\[I \subset \bigcup_i \mathfrak{p}_i\]
	then $I \subset \mathfrak{p}_j$ for some $1\le j \le n$.
\end{lem}

The last lemma is an analogue of the fact that if $V$ is a vector space over an infinite field $\K$ and $W_1,\ldots,W_n$ are proper subspaces, then $\bigcup_i W_i \subsetneq V$ and a generic element of $V$ is not in $\bigcup_i W_i$.

That some property holds for a generic element of a vector space means it holds for a member of a dense open subset of the vector space with respect to an appropriate topology. One such topology here is the Zariski topology obtained by identifying $V$ with an affine space over $\K$.

\begin{rem}[some technicalities]
	The Cohen-Macaulay property has an important role in commutative algebra and algebraic geometry. The discussion here is specialized to the context of face rings: our definition is not the common one outside combinatorics. That it is a specialization of other definitions is a theorem we will not need or prove.
	
	Some amount of dimension theory for commutative rings could not be avoided in the discussion on systems of parameters. Short computational proofs of the necessary facts can be given, and the general case can be found in most textbooks on commutative algebra. We recommend the unfamiliar reader simply accept them for now.
\end{rem}

\subsection{The Koszul Complex}\label[section]{sec:Koszul}
The Koszul complex is a homological tool. It can be used computationally, for instance to find the length of the longest regular sequence contained in an ideal. Conversely, given a regular sequence in $A$ generating an ideal $I$, the Koszul complex gives a free resolution of $A/I$. This is a chain complex of free modules such that $A/I$ is its last (and only nonvanishing) cohomology group. In a sense, it spreads the quotienting operation into simpler layers.

The point of this will become apparent when the partition complex is introduced, and we begin dealing with all stars of faces of a simplicial complex as a cohesive unit. 

Let $A$ be a $\K $-algebra and let $\Theta=(\theta_1,\ldots,\theta_n)$ be a sequence of elements in $A$. The Koszul complex $K^\ast = K^\ast(\Theta)$ is the chain complex
\[0 \rightarrow K^0 \overset{\partial_0}{\rightarrow} K^1 \overset{\partial_1}{\rightarrow} \ldots \rightarrow K^n \overset{\partial_n}{\rightarrow} 0\]
where $K^0 = A$ and in general $K^i = \bigwedge^i A^n = \bigwedge^i \left(\bigoplus_{j=1}^n A\cdot e_j\right).$ The maps $\partial^i:K^i\rightarrow K^{i+1}$ are defined by
\[z\mapsto\left(\sum_{i=1}^{n}\theta_{i}e_{i}\right)\wedge z.\]
This is largely consistent with the notation used by Eisenbud in \cite{eis}.

We collect some basic facts and observations. These culminate in \Cref{thm:Koszul}.

\subsubsection{Exterior powers}
It is helpful to recall that $\bigwedge^i \left(\bigoplus_{j=1}^n A\cdot e_j\right)$ is a free $A$-module with basis all wedges of the form $e_{j_1}\wedge\ldots\wedge e_{j_i}$ where $1\le j_1 < \ldots < j_i\le n$. Thus as $A$-modules,
\[K^i \otimes M = \left(\bigwedge^i A^n\right) \otimes M \simeq \bigoplus_{1\le j_1 <\ldots < j_i\le n}M\cdot e_{j_1}\wedge\ldots\wedge e_{j_i}. \]
The maps of the complex $K^\ast \otimes M$, being induced from $K^\ast$, are defined by the same expression for each $\partial^i$.

\subsubsection{The top cohomology}\label[section]{sec:top_cohomology} Note that $K^n = A\cdot e_1 \wedge \ldots \wedge e_n$ is a free module of rank $1$, and thus $K^n \otimes M \simeq M$. Further, the image of the map $K^{n-1} \otimes M \overset{\partial_{n-1}}{\rightarrow} K^n \otimes M$ is $\langle\Theta\rangle\cdot M$: on generators of $K^{n-1}$ we have 
\[\partial_{n-1}\left(\bigwedge_{j\neq i}e_j\right) =(-1)^{i+1} \theta_i\cdot e_1\wedge\ldots \wedge e_n.\]
This implies $H^n(K^\ast \otimes M) = M/\langle\Theta\rangle M.$

\subsubsection{\texorpdfstring{The action of $z\in \langle\Theta\rangle$ on homology}{Multiplication maps on the homology}}
Suppose $z \in \langle\Theta\rangle$. Then $z$ induces the zero map on $H^i(K^\ast \otimes M)$ for each $i$. To see this, write \[z = \sum_{i=1}^n a_i \theta_i,\] and define the following maps $f^t :K^t  \rightarrow K^{t-1}$:
\[e_{j_1}\wedge\ldots\wedge e_{j_t} \mapsto \sum_{i=1}^t (-1)^{i-1} a_{j_i} \cdot e_{j_1}\wedge\ldots\wedge\hat{e}_{j_i}\wedge\ldots\wedge e_{j_t},\]
where $\hat{e}_{j_i}$ denotes omission. Then the map $\partial^{t-1}\circ f^t + f^{t+1}\circ\partial^t$ acts on $K^t$ as the multiplication map by $z$, and is thus a chain homotopy between $z\cdot$ and the zero map. This remains true on tensoring $K^\ast$ with $M$.

Let us compute $\partial^{i-1}\circ f^i + f^{i+1}\circ\partial^i$ to verify this claim. Without loss of generality, consider the basis element $b = e_1 \wedge\ldots\wedge e_i$. Then
\begin{align*} 
	\partial^{i-1}\circ f^i(b) &= \partial^{i-1}\left(\sum_{j=1}^i (-1)^{j-1} a_j\cdot e_1 \wedge\ldots\wedge \hat{e}_j\wedge\ldots\wedge e_i\right) \\
	&= \sum_{j=1}^i \left[a_j\theta_j\cdot e_1\wedge\ldots\wedge e_i \phantom{\sum_{s=i+1}^n}\right. \\
	& \left. \phantom{= \sum_{j=1}^i \left[\right. }+ \sum_{s=i+1}^n (-1)^{i+j-2}a_j\theta_s\cdot e_1\wedge\ldots\wedge \hat{e}_j\wedge\ldots\wedge e_i\wedge e_s\right].
\end{align*}

Similarly, 
\begin{align*} 
	f^{i+1}\circ \partial^i(b) &= f^{i+1}\left(\sum_{s=i+1}^n (-1)^i\theta_s\cdot e_1\wedge\ldots\wedge e_i\wedge e_s\right) \\
	&= \sum_{s=i+1}^n \left[ \sum_{j=1}^i \right. (-1)^{i+j - 1}a_j\theta_s \cdot e_1\wedge\ldots\wedge \hat{e}_j\wedge\ldots\wedge e_i\wedge e_s \\& \left.\phantom{= \sum_{s=i+1}^n \sum_{j=1}^i } + a_s\theta_s\cdot e_1 \wedge\ldots\wedge e_i\right].
\end{align*}
We see the coefficients cancel in the sum to give $\sum_{j=1}^n a_j\theta_j\cdot e_1\wedge\ldots\wedge e_i = z\cdot e_1\wedge\ldots\wedge e_i$ as desired.

\subsubsection{Koszul homology and regular sequences} We are ready to prove that the homology of the Koszul complex $K^\ast\otimes M$ tells us the maximal length of an $M$-sequence in $\langle\Theta\rangle$.

Recall that the irrelevant ideal of a graded ring is the ideal generated by homogeneous elements of positive degree. For a face ring $\K [\Delta]$, this is simply the ideal $\langle x_v \mid v\in \Delta^{(0)}\rangle$.

\begin{thm}\label[theorem]{thm:Koszul}
	Let $A$ be a graded $\K $-algebra. If $M$ is a finitely-generated graded $A$-module and $\Theta=(\theta_1,\ldots,\theta_n)$ is a sequence of homogeneous elements contained in the irrelevant ideal of $A$, the minimal $i$ for which $H^i(K^\ast(\Theta)\otimes M)\neq 0$ is the maximal length of an $M$-sequence contained in $\langle\Theta\rangle$. This length is at most $n$.
	
	Further, if $\K$ is infinite, generic linear combinations of $\theta_1,\ldots,\theta_n$ give regular sequences.
\end{thm}

\begin{rem}
	We state this theorem in the graded setting, but it is true for Noetherian rings and Noetherian modules in general. The condition that $\Theta$ consists of homogeneous elements contained in the irrelevant ideal should be replaced by the condition $\langle\Theta\rangle\cdot M\neq M$. Generalizing the proof to this situation is an exercise.
\end{rem}

\begin{proof}
	By induction on $i$. For $i=0$, $H^0(K^\ast(\Theta)\otimes M)\neq 0$ iff the map $M \rightarrow \bigoplus_{i=1}^n M\cdot e_i$ given by $m \mapsto \sum_{i=1}^n \theta_i m \cdot e_i$ has a nontrivial kernel, and this occurs iff some $m\in M$ is in the kernel of each multiplication map $\theta_i\cdot: M\to M$ (equivalently, some $m$ is in the kernel of each $z \in \langle \Theta \rangle$).
	
	Suppose the minimal $i$ for which $H^i(K^\ast(\Theta)\otimes M)\neq 0$ is positive, and assume the claim is true up to $i-1$. Then in particular, $H^0(K^\ast(\Theta)\otimes M) =0$, and there is no $m\in M$ in the kernel of all $\langle\Theta\rangle$. Thus $\langle\Theta\rangle$ is not contained in any associated prime (or it would annihilate an element of $M$ by \Cref{lem:associated}). There is then some $z_1\in \langle\Theta\rangle$ which is a nonzerodivisor on $M$, and if $\K$ is infinite we may take $z_1$ to be a linear combination of $\theta_1,\ldots\theta_n$ (a generic one works). Note $M/z_1M \neq 0$, or equivalently $M \neq z_1 M$, since $z_1$ has positive degree, and for the minimal degree $d$ such that $M_d \neq 0$ we have $M_d \cap z_1M = 0$. 
	
	Since each $K^j$ is free over $A$, its tensor product with the short exact sequence
	\[ 0 \to M \overset{z_1\cdot}{\to}M \to M/z_1M \to 0\]
	is exact. This gives a short exact sequence of complexes:
	\[ 0 \to K^\ast(\Theta)\otimes M \overset{z_1\cdot}{\to}K^\ast(\Theta)\otimes  M \to K^\ast(\Theta)\otimes M/z_1M \to 0,\]
	resulting in a long exact sequence in homology:
	\[ \ldots \rightarrow H^j(K^\ast \otimes M) \rightarrow H^j(K^\ast \otimes (M/z_1 M)) \rightarrow H^{j+1}(K^\ast \otimes M) \rightarrow \ldots \]
	where $H^{j+1}(K^\ast \otimes M) = 0$ for $j+1<i$, proving that $H^j(K^\ast \otimes (M/z_1M))=0$ for $j<i-1$. In particular, by the inductive assumption applied to $M/z_1M$, there is an $M/z_1M$-sequence of length $i-1$ in $\langle\Theta\rangle$. Denoting it $z_2,\ldots,z_i$, we have that $z_1,z_2,\ldots,z_i$ is an $M$-sequence.
	
	We will be finished if we prove $H^{i-1}(K^\ast \otimes (M/z_1M))\neq 0$, since applying the inductive claim to $M/z_1M$ then shows that there is no longer $M$-sequence starting with $z_1$ in $\langle\Theta\rangle$, where $z_1$ is an arbitrary nonzerodivisor on $M$.
	
	A piece of the long exact sequence from before reads:
	\[ \ldots \rightarrow 0 \rightarrow H^{i-1}(K^\ast \otimes (M/z_1M)) \rightarrow H^i(K^\ast \otimes M) \overset{z_1\cdot}{\rightarrow} H^i(K^\ast \otimes M)\rightarrow \ldots, \]
	where the first $0$ is $H^{i-1}(K^\ast \otimes M)$. Thus if $H^{i-1}(K^\ast \otimes (M/z_1 M)) = 0$ then $z_1$ is a nonzerodivisor on $H^i(K^\ast \otimes M)\neq 0$. This is a contradiction, since each element of $\langle\Theta\rangle$ acts as the $0$ map on $H^i(K^\ast\otimes M)$.
\end{proof}
\begin{cor}\label[corollary]{cor:CM_choice}
	Under the conditions of the theorem, all maximal $M$-sequences in $\langle\Theta\rangle$ have the same length.
\end{cor}
\begin{proof}
	This follows from the proof above, which picks an arbitrary nonzerodivisor at each step of the induction. There is no possibility of getting stuck in the construction of one sequence earlier than in another.
\end{proof}
\begin{rem}[further notes]
	The explicit chain homotopy used earlier in this section can be replaced by a longer but more illuminating method, which we indicate.
	The Koszul complexes $K^\ast_n=K^\ast(\theta_1,\ldots,\theta_n)$ and $K^\ast_{n+1}=K^\ast(\theta_1,\ldots,\theta_{n+1})$ can be related in the following way: the mapping cone of the map $\theta_{n+1}\cdot:K^\ast_n \rightarrow K^\ast_n$ can be naturally identified with $K^\ast_{n+1}$. Hence there is a short exact sequence of chain complexes 
	\[ 0 \rightarrow K^{\ast-1}_n \rightarrow K^\ast_{n+1} \rightarrow K^\ast_n \rightarrow 0. \]
	If $\theta_{n+1} - \theta'_{n+1} \in \langle \theta_1,\ldots,\theta_n\rangle$, the associated short exact sequences are isomorphic, essentially by a change of basis in each $K^i_{n+1}=\bigwedge^i (A^{n+1})$. In particular, for $\theta_{n+1}\in \langle \theta_1,\ldots,\theta_n\rangle$, one can take $\theta'_{n+1}=0$. The associated long exact sequences are then isomorphic by the naturality of the snake lemma. For $\theta_{n+1}$, the long exact sequence is:
	
	\[\ldots \rightarrow H^{i-1}(K^\ast_n) \rightarrow H^i(K^\ast_{n+1}) \rightarrow H^i(K^\ast_n) \overset{\theta_{n+1} \cdot}\rightarrow H^i(K^\ast_n) \rightarrow \ldots,\]
	and the existence of an isomorphic sequence in which $\theta_{n+1}$ is replaced by $0$ proves the claim. 
\end{rem}

%
%

\section{The Partition Complex \& Reisner's Theorem}\label[section]{sec:Reisner}
In this section we introduce our central tool, the partition complex. As a first application we prove Reisner's theorem, which provides a topological characterization of Cohen-Macaulay complexes.

\subsection{The Partition Complex} In many branches of geometry homological algebra provides tools for piecing together local data and understanding global behavior. In conjunction with the Koszul complex, the partition complex is a tool suitable for understanding face rings and their quotients by systems of parameters.

\begin{defn}
	Let $\Psi=(\Delta,\Gamma)$ be a relative complex. We define the (unreduced) partition complex $P^\ast = P^\ast(\Psi)$ by
	\begin{align*}
	P^\ast = 0\rightarrow \K [\Psi]\overset{d^{-1}}{\rightarrow}\bigoplus_{v\in\Delta^{(0)}}\K [\mathrm{st}_{v}\Psi]\overset{d^{0}}{\rightarrow} 
	\ldots\rightarrow\bigoplus_{\sigma\in\Delta^{(d)}}\K [\mathrm{st}_{\sigma}\Psi]\rightarrow 0, 
	\end{align*}
	with indexing such that $P^{-1}= \K [\Psi]$ and $P^i = \bigoplus_{\tau\in\Delta^{\left(i\right)}}\K [\mathrm{st}_\tau \Psi]$ for $i\ge 0$.
	
	The maps are given by maps of the \v{C}ech complex corresponding to the cover of $(\Delta, \Gamma)$ by the open stars of vertices.
	
	More explicitly, choose some ordering of the vertices. Then if $\rho \subset \tau = \rho\cup\{v_{i_j}\}$ is a pair of faces, and $\tau = \{v_{i_1},\ldots,v_{i_k}\}$ with $i_1 < \ldots < i_k$, then denoting by $\alpha \cdot e_\rho$ the element $\alpha$ in the summand $\K [\mathrm{st}_\rho\Psi]$ (and similarly with $e_\tau$) we define
	\[d(\alpha\cdot e_\rho) = (-1)^j\alpha\cdot e_\tau.\]
	In particular, the map $P^{-1} \rightarrow P^0$ is 
	$\alpha \mapsto \sum_{v\in\Delta^{(0)}}\alpha\cdot e_v.$
\end{defn}

Note that if $x^\alpha$ is a monomial in $\K [\Psi]$, it is a nonvanishing element of $\K [\mathrm{st}_\tau\Psi]$ precisely when the support of $\alpha$ is a face of $\mathrm{st}_\tau(\Delta)$ (it is automatic that it is not a face of $\Gamma$, hence it is not in $\mathrm{st}_\tau(\Gamma)$). The next proposition uses this observation to compute the cohomology of $P^\ast$.

\begin{prop}\label[proposition]{prop:partition_homology}
	For each $i$, $H^i(P^\ast)=H^i(\Psi;\K )$.
\end{prop} 
\begin{proof}
	Each $P^i$ inherits the fine grading of $\K [\Psi]$, and the differential $d^i$ has (fine) degree $0$. Thus, as a complex of vector spaces over $\K $, $P^\ast$ splits into the direct sum of its fine-graded pieces $P^\ast_\alpha$, and this induces a corresponding decomposition of the homology.
	
	Consider a monomial $x^\alpha \in \K [\Psi]$: for any $\tau\in\Delta$, $x^\alpha$ is in $\K [\mathrm{st}_\tau\Psi]$ if and only if $\tau \cup \mathrm{supp}(\alpha)$ is a face of $\Delta$ but not of $\Gamma$, or equivalently if $\tau$ is in the first but not the second member of the pair $\mathrm{st}_{\mathrm{supp}(\alpha)}\Psi$. 
	
	Using this, we see $P^\ast_\alpha$ is a \v{C}ech complex of $\mathrm{st}_{\mathrm{supp}(\alpha)}\Psi$, which is acyclic if $\mathrm{supp}(\alpha)\neq\emptyset$ because stars of faces are contractible. For $x^\alpha = 1$, the complex computes the cohomology of the pair $(\Delta,\Gamma)$.
\end{proof}

The general method is to use this together with the Koszul complex: we form the double complex $P^\ast(\Psi) \otimes K^\ast(\Theta)$ for $\Theta$ some generic sequence of linear forms in $\K [\Delta]$, and perform homological calculations on the resulting double complex.

The same idea works with other complexes of $\K [\Delta]$-modules constructed to compute the cohomology. An example can be seen in \Cref{thm:subdivisions_almost_lefschetz}. Some of this paper can be generalized in such a direction, but we refrain from introducing more formalism at this point.

\subsection{Reisner's Theorem}
Reisner's theorem \cite{Reisner} gives a link between the algebra of face rings and the topology of simplicial complexes: it tells us a complex has a Cohen-Macaulay face ring exactly when it is sufficiently well-connected, both locally and globally. 

\begin{thm}[Reisner]
	Let $\Psi=(\Delta,\Gamma)$ be a relative complex. Then $\K [\Psi]$ is Cohen-Macaulay if and only if for all faces $\tau$ of $\Delta$ (including $\tau=\emptyset$):
	\[{H}^{i}(\mathrm{lk}_{\tau}\Psi;\K )=0\quad\forall -1\le i<\dim(\Psi) - \dim(\tau).\]
\end{thm}

We shall give a new and elementary proof of this theorem here, though for simplicity we restrict to the "if" direction critical for us.

\begin{rem}
	In the non-relative setting, the right-hand side of the last inequality is sometimes replaced by $\dim(\mathrm{lk}_\tau\Delta)$. This is never larger than $\dim(\Delta)-\dim(\tau)$, so our condition is stronger. They are equivalent in the non-relative setting, modulo the topological fact that a (non-relative) complex satisfying the weaker condition is pure. However, this is not true in the relative case. An example is given by 
	
	
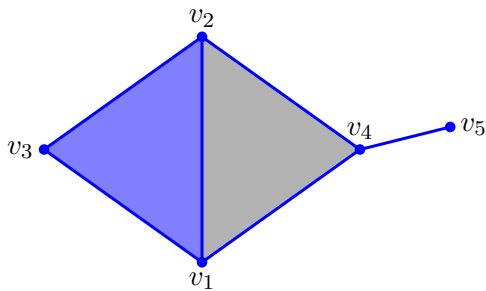
\begin{tikzpicture}
	\coordinate [label=below:$v_1$] (v1) at (0*3,0*3);
	\coordinate [label=above:$v_2$] (v2) at (0*3,1*3);
	\coordinate [label=left:$v_3$] (v3) at (-0.7*3, 0.5*3);
	\coordinate [label=$v_4$] (v4) at (0.7*3, 0.5*3);
	\coordinate [label=right:$v_5$] (v5) at (1.1*3,0.6*3);
	\fill[blue!50] (v1) -- (v2) -- (v3) -- cycle;
	\fill[gray!60] (v1) -- (v2) -- (v4) -- cycle;
	\draw[blue, line width=0.4mm] (v1) -- (v2);
	\draw[blue, line width=0.4mm] (v1) -- (v3);
	\draw[blue, line width=0.4mm] (v3) -- (v2);

	\draw[blue, line width=0.4mm] (v1) -- (v4);
	\draw[blue, line width=0.4mm] (v4) -- (v2);

	\draw[blue, line width=0.4mm] (v4) -- (v5);
	
	\foreach \point in {v1,v2,v3,v4,v5}
	\fill [blue,opacity=1] (\point) circle (2pt);
	
	\end{tikzpicture}
	
	where the complex is relative to the gray (right) triangle.
	
	The stronger assumption ensures purity, because if $\tau \in \Psi$ has lower than maximal dimension then $\mathrm{lk}_\tau\Psi$ cannot have nontrivial cohomology in dimension $-1$, so $\mathrm{lk}_\tau\Psi \neq (\{\emptyset\},\emptyset)$ (note it is automatic that $\mathrm{lk}_\tau(\Gamma)=\emptyset$ since $\tau\in(\Delta,\Gamma)$). Thus $\Psi$ has a face properly containing $\tau$.
\end{rem}
Often, one writes ``$\Delta$ is Cohen-Macaulay" and means whichever of the two equivalent conditions is more convenient for the situation at hand. In this section, to indicate the condition on $\K [\Delta]$ we write that $\Delta$ is algebraically Cohen-Macaulay, and to indicate the homological condition we write that $\Delta$ is topologically Cohen-Macaulay. 

Here is a simple application of Reisner's theorem first mentioned in \Cref{sec:basic_idea}, the importance of which is difficult to overstate: if $M$ is a triangulated manifold, the star of each face is Cohen-Macaulay.

\begin{rem}\label[remark]{Macaulay}
A second, more combinatorial application is of course that of face numbers of simplicial complexes. Notice that the face numbers of a simplicial complex are a non-negative combination of its $h$-numbers, and that these are nonnegative for Cohen-Macaulay complexes by Reisner's theorem (as they correspond to dimensions of vectorspaces).

On the other hand, the $h$-numbers also cannot be too large, as a polynomial ring cannot grow too quickly. For instance, a polynomial ring on $n$ variables cannot generate, in degree $k$, a space of dimension larger than $\binom{n+k-1}{j}$. Macaulay \cite{Macaulay} used this to provide a complete characterization of face numbers of Cohen-Macaulay complexes.
\end{rem}

We begin with the following observations.
\begin{lem}
	If $\Psi$ is topologically Cohen-Macaulay and $\tau\in\Psi$ then $\mathrm{lk}_\tau \Psi$ is topologically Cohen-Macaulay.
\end{lem}
\begin{proof}
	Let $\sigma$ be a face of $\mathrm{lk}_{\tau}\Psi$. We can verify directly that $\mathrm{lk}_{\sigma}(\mathrm{lk}_{\tau}\Psi)$ is the link of $\sigma\cup\tau$ in $\Psi=(\Delta,\Gamma)$: it suffices to verify this in $\Delta$ and in $\Gamma$ separately. So by assumption its homology vanishes beneath the top dimension.
\end{proof}
\begin{lem}\label[lemma]{coning}
	A relative simplicial complex $\Psi=(\Delta,\Gamma)$ is algebraically Cohen-Macaulay if and only if the cone $v\ast\Psi=(v\ast\Delta,v\ast\Gamma)$ is algebraically Cohen-Macaulay.
\end{lem}

\begin{proof}
	The result follows from the fact that $x_v$ is a nonzerodivisor on $M=\K [(v\ast\Delta,v\ast\Gamma)]$ and $M/\langle x_v\rangle M = \K [\Psi] \eqqcolon N.$
	
	In one direction, if $N$ is algebraically Cohen-Macaulay with $\Theta=(\theta_1,\ldots,\theta_n)$ a regular sequence of parameters, then so is $M$ with the sequence $(x_v,\theta_1,\ldots,\theta_n)$. On the other hand, if $M$ is Cohen-Macaulay, then by \Cref{cor:CM_choice} it also has a regular sequence of parameters $(x_v,\theta_1,\ldots,\theta_n)$, and so $N$ is Cohen-Macaulay.
\end{proof}

This is all we need in order to prove the first half of Reisner's theorem.
\begin{thm}
	If $\Psi=(\Delta,\Gamma)$ is topologically Cohen-Macaulay then it is algebraically Cohen-Macaulay.
\end{thm}
\begin{proof}
	By induction on dimension. Consider first a non-relative complex $\Delta$ of dimension $d=0$. Then $\K [\Delta]\simeq \K [x_1,\ldots,x_n]/\langle x_i x_j \mid i \neq j\rangle$ for some $n$, and $\K [\Delta]/\langle \sum_i x_i \rangle$ is finite-dimensional over $\K $: each monomial of degree $2$ vanishes in the quotient, since $x_j ^2 = x_j \cdot \sum_j x_j$. Further, $\sum_i x_i$ is a nonzerodivisor: each nonvanishing element of $\K [\Delta]$ is of the form $c_0 + \sum_i c_i x_i ^{t_i}$ with $c_0,\ldots,c_n\in \K$, and its product by $\sum_i x_i$ does not vanish. 
	
	In the relative case, suppose $d=\dim(\Psi)=0$. Then unless $\Gamma=\emptyset$ (the non-relative situation) we have $\K [\Psi]=\langle x_v \mid v\in\Psi \rangle$, and again quotienting by the ideal generated by $\sum_{v\in\Psi} x_v$ (a nonzerodivisor) gives a finite dimensional vector space.
	
	Suppose the theorem holds for relative complexes of dimension up to $d-1$, and suppose $\Psi$ is topologically Cohen-Macaulay of dimension $d$. Then the link of each face $\tau$ is algebraically Cohen-Macaulay of dimension $d - \dim(\tau)$, and by the coning lemma $\mathrm{st}_\tau\Psi$ is algebraically Cohen-Macaulay as well.
	
	Let $\Theta$ be a generic sequence of $d+1$ linear forms in $\K [\Delta]$, and consider $C^{\ast,\ast}=P^\ast(\Psi) \otimes K^\ast(\Theta)$. By \Cref{Tot_mapping_cone} the mapping cone of $\mathfrak{R}^{-1}$ gives us a short exact sequence
	\[0\rightarrow \mathrm{Tot}(C^{\ast\ge 0, \ast})^{\ast-1} \rightarrow \mathrm{Tot}(C^{\ast\ge -1, \ast})^{\ast-1}=\mathrm{Tot}(C^{\ast,\ast})^{\ast-1} \rightarrow C^{-1,\ast}\rightarrow 0.\]
	
	The corresponding long exact sequence is
	\begin{align*} 
	 \ldots & \rightarrow H^i(C^{-1,\ast}) \rightarrow H^i(\mathrm{Tot}(C^{\ast\ge 0, \ast})^\ast) \rightarrow H^i(\mathrm{Tot}(C^{\ast,\ast})^\ast) \rightarrow H^{i+1}(C^{-1,\ast}) \rightarrow 
	 \\
	 \ldots & \rightarrow H^{d-1}(\mathrm{Tot}(C^{\ast,\ast})^\ast)
	  \rightarrow H^d(C^{-1,\ast}) \rightarrow H^d(\mathrm{Tot}(C^{\ast\ge 0, \ast})^\ast)  \rightarrow \ldots
	\end{align*}
	Notice that if $k < d$ then $H^{k-i}(C^{\ast,i})=0$: for $i<0$ this is automatic because $C^{\ast,i}$ is the $0$ complex, and otherwise it occurs because the row $C^{\ast,i}=P^\ast \otimes K^i(\Theta)$ is exact until the $d$-th place, since $P^\ast$ is (and $K^i$ is a free module). By \Cref{lem:exactness}, this implies $H^k(\mathrm{Tot}(C^{\ast,\ast})^\ast)=0$ for $k<d$. 
	
	This implies the maps $C^{-1,\ast}\rightarrow\mathrm{Tot}(C^{\ast\ge0,\ast})^\ast$ induce isomorphisms on the $k$-th homology for all $k<d$, and an injection on $H^{d}$.
	
	Since the star of each face is algebraically Cohen-Macaulay, each column of $C^{\ast\ge 0,\ast}$ has 
	$H^i(C^{j,\ast})=0$ for all $i<d+1$, so another application of \Cref{lem:exactness} proves $H^i(\mathrm{Tot}(C^{\ast\ge0,\ast})^\ast)=0$ for $i<d+1$ as well. This implies that $H^i(C^{-1,\ast})=H^i(\K[\Psi]\otimes K^\ast(\Theta))=0$ for all such $i$, and the theorem follows. 
\end{proof}

%

\section{Partition of Unity}

Let $\Psi=(\Delta,\Gamma)$ be a triangulated $d$-manifold
with boundary, and let $\Gamma=\emptyset$ or the boundary subcomplex
(that the boundary always is in fact a subcomplex is a simple result
in topology). Then for each $\tau\in\Delta$, $\mathrm{st}_{\tau}\Psi$
is Cohen-Macaulay of dimension $d$. This is the situation in which
the partition complex is most useful: computations of interest can
be reduced to the homology of $\mathrm{Tot}(P^\ast\otimes K^\ast)$.
To motivate the computation of this homology we begin with an application,
the partition of unity theorem.

For a generic sequence of linear elements $\Theta=(\theta_{1},\ldots,\theta_{d+1})$
in $\K [\Delta]$, we wish to understand $P^\ast(\Psi)/\left\langle \Theta\right\rangle $.
The partition of unity theorem computes the dimension of the kernel
of the first differential of this complex, i.e. of
\[
\K [\Psi]/\left\langle \Theta\right\rangle \rightarrow\bigoplus_{v\in\Delta^{(0)}}\K [\mathrm{st}_{v}\Psi]/\left\langle \Theta\right\rangle ,
\]
in terms of $H^\ast(\Delta,\Gamma)$. 

Here is a special case: if $(\Delta,\emptyset)$ is a sphere
then this map is injective except in the top degree $d+1$, where
the kernel has dimension one (and in fact equals the entire $(d+1)$-th
graded piece). 

We state the theorem in slightly greater generality, recovering a key result of \cite{AHL}.
\begin{thm}
	[partition of unity] Let $\Psi=(\Delta,\Gamma)$ be a
	relative complex, pure of dimension $d$, such that $\mathrm{st}_{\tau}\Psi$
	is Cohen-Macaulay for each nonempty $\tau\in\Delta$. Let $\Theta$
	be a generic sequence of $d+1$ elements in $\K [\Delta]_{1}$.
	Then
	\[
	\dim_{\K }H^{i}(P^\ast/\left\langle \Theta\right\rangle )_{j}={d+1 \choose j}\cdot\dim_{\K }H^{i+j}(\Psi;\K ).
	\]
\end{thm}
Complexes satisfying the conditions of the theorem are called Buchsbaum complexes.
Note the subscript $j$ indicates the $j$-th graded piece is being
taken, where the grading is induced from $\K [\Delta]$. 
\begin{proof}
	Recall $H^{d+1}(\Theta)=\K [\Psi]/\left\langle \Theta\right\rangle $:
	this is shown in \Cref{sec:top_cohomology}. Define $\tilde{K}^\ast(\Theta)$ to be the \emph{augmented
		Koszul complex}.
	This is the complex
	\[
	K^{0}(\Theta)\rightarrow\ldots\rightarrow K^{d+1}(\Theta)\rightarrow\tilde{K}^{d+2}(\Theta)=H^{d+1}(\Theta)=\K [\Delta]/\left\langle \Theta\right\rangle \rightarrow0,
	\]
	where the map $\tilde{K}^{d+1}(\Theta)\rightarrow H^{d+1}(\Theta)$
	is the natural quotient map $K^{d+1}\rightarrow\nicefrac{K^{d+1}}{\mathrm{im}(K^{d})}$.

	We will work with the complex $C^{\ast,\ast}=P^\ast\otimes\tilde{K}^\ast(\Theta)$.
	
	Since the last nonzero map of $\tilde{K}^\ast$ is the cokernel of the map before it, $H^{d+1}(\tilde{K}^\ast\otimes M)=H^{d+2}(\tilde{K}^\ast\otimes M)=0$
	for any $\K [\Delta]$-module $M$. Therefore the Cohen-Macaulayness
	of the stars of $\Psi$ implies $\tilde{K}^\ast(\Theta)\otimes \K [\mathrm{st}_{\tau}\Psi]$
	is exact for each $\tau\in\Delta$ (homologies up to the $d$-th position
	are equal for the augmented and un-augmented complex). In particular,
	the columns $C^{i,\ast}$ are exact for each $i\ge0$. The column $C^{-1,\ast}=\K[\Psi]\otimes \tilde{K}^\ast(\Theta)$
	is only exact if $\Psi$ is Cohen-Macaulay. 
	
	Consider the map $\mathfrak{U}^{d+1}:\mathrm{Tot}(C^{\ast,\ast\le d+1})^\ast\rightarrow C^{\ast-d-1,d+2}$.
	The exact sequence associated to its mapping cone is
	\[
	0\rightarrow C^{\ast-d-2,d+2}\rightarrow\mathrm{Tot}(C^{\ast,\ast\le d+2})^\ast=\mathrm{Tot}(C^{\ast,\ast})^\ast\rightarrow\mathrm{Tot}(C^{\ast,\ast\le d+1})^\ast\rightarrow0.\tag{{*}}
	\]
	Note that $H^{d-i}(C^{\ast,i})=H^{d+1-i}(C^{i,\ast})=0$
	for all $i$. For $i\ge0$ this is automatic, since the entire $i$-th
	column is exact. For $i=-1$, it follows from the fact that $H^{d+1}(C^{-1,\ast})=H^{d+2}(C^{-1,\ast})=0$,
	as both are equal to the respective homologies of $\tilde{K}^\ast$.
	Therefore \Cref{lem:exactness} implies 
	\[
	H^{d}(\mathrm{Tot}(C^{\ast,\ast})^\ast)=H^{d+1}(\mathrm{Tot}(C^{\ast,\ast})^\ast)=0.
	\]
	Similarly, $H^{t}(\mathrm{Tot}(C^{\ast,\ast})^\ast)$
	for all $t>d+1$, since the $(-1)$-th column is not involved.
	Therefore for all $t\ge d$ the short sequence $(*)$ above
	yields the exact sequences in homology
	\[
	\ldots0\rightarrow H^{t}(\mathrm{Tot}(C^{\ast,\ast\le d+1})^\ast)\rightarrow H^{t+1}(C^{\ast-d-2,d+2})\rightarrow0,
	\]
	where the zeros on the left and right end are the $t$ and $(t+1)$-th
	homologies of $\mathrm{Tot}(C^{\ast,\ast})^\ast$ respectively. 
	
	Set $s=i+d+2$. Since the row $C^{\ast,d+2}$ is $P^\ast/\left\langle \Theta\right\rangle $,
	$H^{s}(C^{\ast-d-2,d+2})=H^{i}(P^\ast/\left\langle \Theta\right\rangle )$.
	We now know it is isomorphic to 
	\[
	H^{s-1}(\mathrm{Tot}(C^{\ast,\ast\le d+1})^\ast)=H^{i+d+1}(\mathrm{Tot}(C^{\ast,\ast\le d+1})^\ast),
	\]
	where $C^{\ast,\ast\le d+1}\simeq P^\ast\otimes K^\ast(\Theta)$.
	The claim is now reduced to \Cref{thm:total_partition_complex}. From
	its statement we conclude the identity
	\begin{align*}
	\dim_{\K }H^{i}(P^\ast/\left\langle \Theta\right\rangle )_{j} &=\dim_{\K }H^{i+d+1}(\mathrm{Tot}(P^\ast\otimes K^\ast)^\ast)_{j} \\
	&={d+1 \choose j}\cdot\dim_{\K }H^{i+j}(\Psi;\K ).
	\end{align*}
\end{proof}

\subsection{{\texorpdfstring{The homology of $\mathrm{Tot}(P^\ast\otimes K^\ast)$}{Homology of the total complex}}}

As earlier, $\Psi=(\Delta,\Gamma)$ is relative complex,
pure of dimension $d$, in which the star of each $\tau\in\Delta$
is Cohen-Macaulay. Let $\Theta=(\theta_{1},\ldots,\theta_{d+1})$
a generic sequence in $\K [\Delta]_{1}$ and denote $C^{\ast,\ast}=P^\ast(\Psi)\otimes K^\ast(\Theta)$.

\subsubsection{Grading the double complex}

We wish to understand the homology of $\mathrm{Tot}(C^{\ast,\ast})$
as a graded $\K [\Delta]$-module. It is convenient to grade
each $C^{i,j}$ such that the differentials of the double complex
are maps of graded modules, or in other words take homogeneous elements
to homogeneous elements of the same degree.

We grade $K^\ast$ as follows. Each $K^{i}$ has a basis of the form
\[\left\{ e_{j_{1}}\wedge\ldots e_{j_{i}}\mid j_{1}<j_{2}<\ldots<j_{i}\right\} .\]
Set each such basis element to be of degree $d+1-i$. In particular,
$K^{d+1}\simeq \K [\Delta]$ is a free module with basis
$e_{1}\wedge e_{2}\wedge\ldots\wedge e_{d+1}$, and this basis element
has degree $0$. Similarly, $K^{0}=\K [\Delta]$ is a free
module with basis $\left\{ 1\right\} $, re-graded so that $1$ has
degree $d+1$ (the common notation for this is $K^{0}(\Theta)=\K [\Delta](-d-1)$).

Thus if $\alpha\in \K [\Delta]_{t}$ is homogeneous of degree
$t$, $\alpha\cdot e_{j_{1}}\wedge\ldots\wedge e_{j_{i}}$ has degree
$t+d+1-i$. The differential, $z\mapsto(\sum\theta_{i}e_{i})\wedge z$,
is then a map of graded modules. This induces the required grading
on $P^\ast\otimes K^\ast$. More generally it induces a grading on $K^\ast\otimes M$
for $M$ any graded $\K [\Delta]$-module, in which $m\cdot e_{j_{1}}\wedge\ldots\wedge e_{j_{i}}$
has degree $t+d+1-i$ for any $m\in M_{t}$.

Since the kernel, image, and cokernel of a map of graded modules are
graded modules, there is an induced grading on the homology of the
rows, columns, and total complex of $C^{\ast,\ast}$. In fact, if we only
want to compute dimensions, these can be computed separately for each
graded piece of $C^{\ast,\ast}$ (no longer a $\K [\Delta]$-module
but a $\K $-vector space).

\subsubsection{{\texorpdfstring{The grading and $P^\ast/\left\langle \theta\right\rangle $}{The grading and the reduced partition complex}}}

Our choice of grading for $K^\ast$ is motivated by the following consideration.
The augmented complex $\tilde{K}^\ast$ has $\tilde{K}^{d+1}\simeq \K [\Delta]$
(the isomorphism now being of graded modules), and the image of $\tilde{K}^{d}$
in $\tilde{K}^{d+1}$ is precisely $\left\langle \Theta\right\rangle $.
Thus the cokernel $\tilde{K}^{d+2}=H^{d+1}(K^\ast)$ of
$\tilde{K}^{d}\rightarrow\tilde{K}^{d+1}$ is precisely $\K [\Delta]/\left\langle \Theta\right\rangle $,
and has the same grading.

These considerations apply just as well when $\tilde{K}^\ast$ is replaced
with $\tilde{K}^\ast\otimes M$ for $M$ a graded $\K [\Delta]$-module.
Hence all computations on $P^\ast(\Psi)/\left\langle \Theta\right\rangle $
performed using these methods respect its grading automatically.

\subsubsection{Computation of the homology}

With the above grading in hand, we can compute $H^{i}(\mathrm{Tot}(C^{\ast,\ast}))_{j}$
for any $i,j$.
\begin{thm}
	\label[theorem]{thm:total_partition_complex}For any $i$ and $j$,
	we have \[\dim_{\K }H^{i}(\mathrm{Tot}(C^{\ast,\ast}))_{j}={d+1 \choose j}\cdot\dim_{\K }H^{i+j-d-1}(\Psi;\K ).\]
\end{thm}
\begin{proof}
	As often happens, the proof writes itself once the correct short exact
	sequence of complexes is found. In the interest of making this process
	as transparent as possible, let us begin by explicitly unraveling
	the degrees and indices involved. 
	
	Observe that for $j\in\mathbb{Z}$ any module in the row $C_{j}^{\ast,s}$
	is spanned by elements of the form $\alpha\cdot e_{j_{1}}\wedge\ldots\wedge e_{j_{s}}$,
	where $\alpha\in \K [\mathrm{st}_{\tau}\Psi]$ and
	\[
	\deg(\alpha)+d+1-s=j,
	\]
	or $\deg(\alpha)=j+s-d-1$.
	
	We now recall \Cref{prop:partition_homology}. Its proof shows $P_{t}^\ast$
	is exact for all $t>0$.
	This means the row $C_{j}^{\ast,s}$ is exact unless $d+1-j-s=0$, so
	the interesting row is $s=d+1-j$. In any lower row $s'<s$, the complex
	$C_{j}^{\ast,\ast}$ vanishes entirely: elements there are sums of expressions
	$\alpha\cdot e_{j_{1}}\wedge\ldots\wedge e_{j_{s'}}$, where $\alpha\in \K [\mathrm{st}_{\tau}\Psi]$
	has negative degree. That the lowest row is also the only one carrying
	nontrivial homology suggests relating $H^\ast(C^{\ast,s})_{j}$
	with $H^\ast(\mathrm{Tot}(C^{\ast,\ast\ge s})^\ast)_{j}$
	and with $H^\ast(\mathrm{Tot}(C^{\ast,\ast\ge s+1})^\ast)_{j}=0$. 
	
	Taking the transpose $(\mathfrak{R}^{i})^{\dagger}$ of
	$\mathfrak{R}^{i}$ (exchanging the two upper indices, along with
	the vertical and horizontal differentials) provides a map
	\[
	(\mathfrak{R}^{\dagger})^{s}:C_{j}^{\ast,s}\rightarrow\mathrm{Tot}(C^{\ast,\ast\ge s+1})_{j}^{\ast+s+1},
	\]
	with mapping cone $\mathrm{Tot}(C^{\ast,\ast\ge s})_{j}^{\ast+s}$.
	Thus there is a long exact sequence in homology:
	\[
	\ldots\rightarrow H^i(\mathrm{Tot}(C^{\ast,\ast\ge s+1})_{j}^{\ast+s})\rightarrow H^{i}(\mathrm{Tot}(C^{\ast,\ast\ge s})_{j}^{\ast+s})\rightarrow
	\]
	\[
	H^{i}(C_{j}^{\ast,s})\rightarrow H^{i+1}(\mathrm{Tot}(C^{\ast,\ast\ge s+1})_{j}^{\ast+s})\rightarrow\ldots
	\]
	where $\mathrm{Tot}(C^{\ast,\ast\ge s+1})_{j}^{\ast+s}$ is exact.
	We obtain isomorphisms 
	\[
	H^{i+s}(\mathrm{Tot}(C^{\ast,\ast})_{j}^\ast)=H^{i}(\mathrm{Tot}(C^{\ast,\ast\ge s})_{j}^{\ast+s})\simeq H^{i}(C_{j}^{\ast,s}),
	\]
	or
	\[
	H^{i}(\mathrm{Tot}(C^{\ast,\ast})_{j}^\ast)=H^{i-s}(C_{j}^{\ast,s}).
	\]
	Since $s=d+1-j$, the right hand side is \[H^{i+j-d-1}(P^\ast\otimes K^{d+1-j})=H^{i+j-d-1}(P^\ast)\otimes K^{d+1-j},\]
	which we know by \Cref{prop:partition_homology} to have dimension
	\[
	{d+1 \choose j}\cdot\dim_{\K }H^{i+j-d-1}(\Psi;\K )
	\]
	as claimed.
\end{proof}

\section{Schenzel's Formula}\label[section]{sec:Schenzel}

We now know enough to compute $\dim_{\K }(\K [\Psi]/\left\langle \Theta\right\rangle )_{j}$
for $\Psi$ a $d$-dimensional triangulated manifold with boundary. This is a rather powerful fact, and we recover a formula of Schenzel \cite{Schenzel81, Schenzel82}, which allows one to generalize the characterization of face numbers to manifolds and Buchsbaum complexes.

As in the previous section, $\Gamma$ may be the boundary subcomplex
or $\emptyset$, independently of whether $\Delta$ has a nonempty
boundary. The double complex $C^{\ast,\ast}$ is $P^\ast(\Psi)\otimes K^\ast(\Theta)$.

Recall that if $\Psi=(\Delta,\Gamma)$ is Cohen-Macaulay,
$\dim_{\K }(\K [\Psi]/\left\langle \Theta\right\rangle )_{j}=h_{j}(\Psi)$
is an entry of the $h$-vector, and hence determined by the $f$-vector.
For general $\Psi$ this is not the case. To perform the computation
we use the map $\mathfrak{R}^{-1}:C^{-1,\ast}\rightarrow\mathrm{Tot}(C^{\ast\ge0,\ast})^\ast$
as in the proof of Reisner's theorem, obtaining once again the long
exact sequence \begin{align*}  	 \ldots & \rightarrow H^i(C^{-1,\ast}) \rightarrow H^i(\mathrm{Tot}(C^{\ast\ge 0, \ast})^\ast) \rightarrow H^i(\mathrm{Tot}(C^{\ast,\ast})^\ast) \rightarrow H^{i+1}(C^{-1,\ast}) \rightarrow  	 \\ 	 \ldots & \rightarrow H^{d-1}(\mathrm{Tot}(C^{\ast,\ast})^\ast) 	  \rightarrow H^d(C^{-1,\ast}) \rightarrow H^d(\mathrm{Tot}(C^{\ast\ge 0, \ast})^\ast)  \rightarrow \ldots 	\end{align*}We
know $H^{i}(\mathrm{Tot}(C^{\ast\ge0,\ast})^\ast)=0$
for each $i\le d$, since columns with nonnegative indices are exact
until the $(d+1)$-th place. This gives isomorphisms
\[
H^{i-1}(\mathrm{Tot}(C^{\ast,\ast})^\ast)_{j}\simeq H^{i}(C^{-1,\ast})_{j}
\]
for each $i\le d$, so
\[
\dim_{\K }H^{i}(C^{-1,\ast})_{j}={d+1 \choose j}\cdot\dim_{\K }H^{i+j-d-2}(\Psi;\K )
\]
for such $i$. Now, it is a general fact that
\[
\sum_{i=0}^{d+1}(-1)^{d+1-i}\dim_{\K }H^{i}(C^{-1,\ast})_{j}=\sum_{i=0}^{d+1}(-1)^{d+1-i}\dim_{\K }(C^{-1,i})_{j},\tag{{*}}
\]
these being two expressions for the Euler characteristic of the complex,
multiplied by $(-1)^{d+1}$. The left hand side of $(*)$
is
\[
\dim_{\K }H^{d+1}(C^{-1,\ast})_j+\sum_{i=0}^{d}(-1)^{d+1-i}\dim_{\K }H^{i}(C^{-1,\ast})_{j}
\]
\[
=\dim_{\K }(\K [\Delta,\Gamma]/\left\langle \Theta\right\rangle )_{j}+{d+1 \choose j}\sum_{i=0}^{d}(-1)^{d+1-i}\dim_{\K }H^{i+j-d-2}(\Delta,\Gamma;\K)
\]
\[
=\dim_{\K }(\K [\Delta,\Gamma]/\left\langle \Theta\right\rangle )_{j}+{d+1 \choose j}\sum_{i=0}^{j-2}(-1)^{i+j+1}\dim_{\K }H^{i}(\Delta,\Gamma;\K)
\]
by our homological computations above. Note that in the Cohen-Macaulay
case this is just $\dim_{\K }(\K [\Psi]/\left\langle \Theta\right\rangle )_{j}$.

To compute the right hand side of $(*)$, note each $C^{-1,i}=\K [\Psi]\otimes K^{i}(\Theta)$
is fully known. In fact the alternating sum yields $h_{j}(\Psi)$:
this is easiest to compute using the fact that the Hilbert series
of $C^{-1,i}$ is 
\[
{d+1 \choose i}x^{d+1-i}H_{\K [\Psi]}(x).
\]
Substituting, we find
\[
\dim_{\K }(\K [\Psi]/\left\langle \Theta\right\rangle )_{j}=h_{j}(\Psi)+{d+1 \choose j}\sum_{i=0}^{j-2}(-1)^{i+j}\cdot\dim_{\K }H^{i}(\Psi;\K ),
\]
a formula originally due to Schenzel.
\begin{example}
	The unique minimal triangulation $\Delta$ of the torus has $7$ vertices,
	$21$ edges, and $14$ triangles. Thus the Hilbert series of its
	face ring is
	\[
	\frac{(1-x)^{3}+7x(1-x)^{2}+21x^{2}(1-x)+14x^{3}}{(1-x)^{3}}=\frac{1+4x+10x^{2}-x^{3}}{(1-x)^{3}},
	\]
	and the $h$-vector is $(1,4,10,-1)$. Recall the dimensions
	of the cohomology groups are $0,2,1$ in dimensions $0,1,2$ respectively.
	Schenzel's formula tells us that for $\Theta$ a generic linear system
	of parameters, the dimensions of the graded pieces of $\K [\Delta]/\left\langle \Theta\right\rangle $
	are $1,\ 4,\ 10$ and $1$.
\end{example}

\section{Poincar\'{e} duality}\label[section]{sec:Poincare}

Let $\Delta$ be a closed, connected, and orientable triangulated
manifold, $\Theta$ a linear system of parameters. We shall prove
a certain ring associated with $\Delta$ is a Poincar\'e duality
algebra: in the special case of $\Delta$ a sphere, this algebra is
just $\K[\Delta]/\langle \Theta\rangle $. 

The theorem requires a little preparation. 

\subsection{Poincar\'e duality algebras in general}
\begin{defn}
	A finitely generated graded $\K$-algebra $A$ is a Poincar\'e duality
	algebra of degree $n$ if:
	\begin{enumerate}
		\item $A_{i}=0$ unless $0\le i\le n$,
		\item $A_{n}\simeq \K$, 
		\item For any $0\le i\le n$ the multiplication map induces a non-degenerate
		bilinear pairing
		\[
		A_{i}\times A_{n-i}\rightarrow A_{n}.
		\]
	\end{enumerate}
\end{defn}
The last statement means that for any nonzero $x\in A_{i}$ there
exists a $y\in A_{n-i}$ such that $xy\in A_{n}$ is nonzero. This
implies that $A_{i}\simeq A_{n-i}$ as vector spaces over $\K$.

The case in which $A$ is generated in degree $1$ is particularly
nice.
\begin{lem}
	\label[lemma]{lem:Poincare_deg_1}Let $A$ be a finitely generated
	graded $\K$-algebra generated in degree $1$. Then the following conditions
	are equivalent.
	\begin{enumerate}
		\item $A$ is a Poincar\'e duality algebra of degree $n$.
		\item $A$ vanishes above degree $n$, and $\left\{ a\in A\mid ax=0\text{ for all }x\in A_{1}\right\} =A_{n}$.
	\end{enumerate}
\end{lem}
\begin{proof}
	If $A$ is a Poincar\'e duality algebra of degree $n$ and $i<n$,
	let $a\in A_{i}$. There is a $y\in A_{n-i}$ such that $ay\neq0$.
	Since $A$ is generated in degree $1$, $y$ is an expression in degree-$1$
	elements, and one of these has a nonzero product with $a$. Thus for
	any $x\in A$, if $xA_{1}=0$ then the degree-$i$ part of $x$ is
	zero for all $i<n$, or in other words $x\in A_{n}$.
	
	If $A$ vanishes above degree $n$ and $\left\{ a\in A\mid ax=0\text{ for all }x\in A_{1}\right\} =A_{n}$,
	let us show each $x\in A_{j}$ has some $y\in A_{n-j}$ such that
	$xy\neq0$ by descending induction on $j$. For $j=n-1$, this is
	just the assumption: if $x\in A_{n-1}$ has $xy=0$ for all $y\in A_{1}$
	then $x\in A_{n}$, but then $x\in A_{n-1}\cap A_{n}=0$.
	
	Suppose the statement is known for some $j\ge1$ and let $x\in A_{j-1}$.
	Then again there is some $y\in A_{1}$ such that $xy\neq0$ as before,
	and $xy\in A_{j}$. Thus there is a $z\in A_{n-j}$ such that $(xy)z\neq0$.
	The product $x(yz)$ is nonzero and in $A_{n}$.
\end{proof}

\subsection{Poincar\'e duality for face rings of manifolds}\label[section]{Poincare}

Let $\Delta$ be a triangulation of a closed, connected, orientable
manifold of dimension $d$, and let $\Theta=(\theta_{1},\ldots,\theta_{d+1})$
be a generic sequence in $\K[\Delta]_{1}$.

Observe that $\K[\Delta]/\langle \Theta\rangle $
vanishes above degree $d+1$, and the degree $d+1$ part has dimension
$1$, by Schenzel's formula and the fact that $h_{d+1}(\Delta)=(-1)^{d}\chi(\Delta)$.
However, $\K[\Delta]/\langle \Theta\rangle $
is generally not a Poincar\'e duality algebra. For instance, above
we used Schenzel's formula to show $\dim \K[\Delta]_{1}\neq\dim \K[\Delta]_{2}$
for $\Delta$ the minimal triangulation of a torus. 

For convenience, we denote
\[
A\coloneqq A(\Delta)\coloneqq \K[\Delta]/\langle \Theta\rangle ,
\]
\[
A(\mathrm{st}_{v}\Delta)\coloneqq \K[\mathrm{st}_{v}\Delta]/\langle \Theta\rangle ,
\]
and similarly $A(\mathrm{st}_{v}^{\circ}\Delta)=\K[\mathrm{st}_{v}\Delta,\mathrm{lk}_{v}\Delta]/\langle \Theta\rangle $.

Let us cut straight to the point of the idea, so as not to lose the
forest for the trees. Details are provided further in this section.
The goal is to introduce a nice quotient $B=B(\Delta)$
of $A$ in which Poincar\'e duality holds. Its defining properties
are the following: 
\begin{enumerate}
	\item \label[property]{property:prop_1}$B$ is a graded quotient of $A$,
	and $B_{d+1}\simeq A_{d+1}\simeq \K$.
	\item \label[property]{property:prop_2}For each $v\in\Delta^{(0)}$
	the restriction $A(\Delta)\rightarrow A(\mathrm{st}_{v}\Delta)$
	factors through $B$. In particular restriction maps $B\rightarrow A(\mathrm{st}_{v}\Delta)$
	are defined.
	\item \label[property]{property:prop_3}For each $\alpha\in B$ of degree
	less than $d$, there is some $v\in\Delta^{(0)}$ such
	that the restriction of $\alpha$ to $A(\mathrm{st}_{v}\Delta)$
	is nonzero. 
	\item \label[property]{property:prop_4}For each $v\in\Delta^{(0)}$
	the composition $A(\mathrm{st}_{v}^{\circ}\Delta)\rightarrow A\rightarrow B$
	is injective.
\end{enumerate}
Once $B$ is constructed and the properties above are shown to hold,
the theorem is rather short. We need one last, basic lemma.
\begin{lem}
	[the cone lemma] For each $v\in\Delta^{(0)}$, there is
	an isomorphism
	\[
	A(\mathrm{st}_{v}\Delta)\rightarrow A(\mathrm{st}_{v}^{\circ}\Delta)
	\]
	\[
	\alpha\mapsto x_{v}\alpha.
	\]
\end{lem}
\begin{proof}
	First consider $\mathrm{lk}_{v}\Delta$ as a subcomplex of $\mathrm{st}_{v}\Delta$.
	Each nonface of the link contains the vertex $v$, so the non-face
	ideal of the link is just $\langle x_{v}\rangle $. Thus
	$\K[\mathrm{st}_{v}\Delta,\mathrm{lk}_{v}\Delta]=\langle x_{v}\rangle /I_{\mathrm{st}_{v}\Delta}$,
	and the map
	\[
	\K[\mathrm{st}_{v}\Delta]\rightarrow\K[\mathrm{st}_{v}\Delta,\mathrm{lk}_{v}\Delta]
	\]
	\[
	\alpha\mapsto x_{v}\cdot\alpha
	\]
	is an isomorphism, since it is surjective and injective, $x_{v}$
	being a nonzerodivisor in $\K[\mathrm{st}_{v}\Delta]$. The lemma then
	follows since $A(\mathrm{st}_{v}\Delta)$ and $A(\mathrm{st}_{v}^{\circ}\Delta)$
	are quotients of these two isomorphic modules by their product with
	the same ideal $\langle \Theta\rangle $.
\end{proof}

The following goes back to unpublished early work of Hochster \cite{Stanley96}, with partial results in a series of papers of Novik and Swartz \cite{NS2}.

\begin{thm}[Poincar\'e duality]\label[theorem]{thm:pd}
	 The ring $B(\Delta)$ is a Poincar\'e
	duality algebra.
\end{thm}
\begin{proof}
	Let $\alpha\in B_{i}$ for $i<d+1$. Then by \Cref{property:prop_3} of $B$
	there is a $v\in\Delta^{(0)}$ such that the restriction
	of $\alpha$ to $A(\mathrm{st}_{v}\Delta)$ is nonzero.
	Consider the maps
	\[
	B\rightarrow A(\mathrm{st}_{v}\Delta)\rightarrow A(\mathrm{st}_{v}^{\circ}\Delta)\rightarrow B,
	\]
	and denote their composition by $f$. Observe $f(\alpha)\neq0$
	is nonzero: the restriction of $\alpha$ to $A(\mathrm{st}_{v}\Delta)$
	is nonzero by assumption. The next map, to $A(\mathrm{st}_{v}^{\circ}\Delta)$,
	has no kernel because it is an isomorphism. The last map is injective
	by \Cref{property:prop_4}. 
	
	Since $f$ is a map of $\K[\Delta]$-modules,
	it is determined by the image of the generator $1$ of $B$.
	This image is $x_{v}$ by definition. Hence $f$ is just the multiplication
	map by $x_{v}$, and $x_{v}\alpha\neq0$. Since $B$ is generated
	in degree $1$ (it is a quotient of $A$, hence of $\K[\Delta]$),
	the theorem follows from \Cref{lem:Poincare_deg_1}.
\end{proof}
\begin{rem}
	Our notation $A,B$ for the rings in this section follows \cite{AHL}.
	However, that paper uses upper indices for the grading (lower indices
	take a different role there, denoting a dual object). We denote the
	grading with lower indices, as in the rest of this paper.
\end{rem}

\subsubsection{\texorpdfstring{The algebra $B(\Delta)$ and its properties}{The algebra B and its properties}}

In this section we contruct $B(\Delta)$ for $\Delta$
a triangulated manifold as above and prove the necessary properties
hold. The reader may wish to follow the proof for the case of $\Delta$
a sphere at first. In this case, $B(\Delta)=A(\Delta)$
and the only missing piece is \Cref{property:prop_4}, proved below in \Cref{prop:ost_A_inclusion}.
\begin{defn}
	The ideal $J=J(\Delta)\subset A(\Delta)$ is
	$0$ in degree $d+1$, and in degrees $i<d+1$ it is
	\[
	H^{-1}(P^\ast/\langle \Theta\rangle )_{i}=\ker\left[A(\Delta)\rightarrow\bigoplus_{v\in\Delta^{(0)}}A(\mathrm{st}_{v}\Delta)\right]_{i}.
	\]
	We define $B(\Delta)=A(\Delta)/J$.
\end{defn}
To show this is well defined, we need to prove $J$ is in fact an
ideal. Clearly it is a vector subspace in each degree. The product
of an element in the kernel of the restriction map to all vertex stars
and another element of $A(\Delta)$ is clearly again in
the kernel, so all that remains to show is that no such product has
a nonzero component in degree $d+1$. We show something stronger.
\begin{lem}
	Let $\alpha\in A$ such that the restriction of $\alpha$ to $A(\mathrm{st}_{v}\Delta)$
	vanishes for each $v\in\Delta^{(0)}$, and let $\beta\in A$
	be homogeneous of degree at least $1$. Then $\alpha\beta=0$.
\end{lem}
\begin{proof}
	It suffices to prove the lemma for $\beta$ of degree $1$, since
	any homogeneous element of positive degree is a polynomial expression
	in such elements. Any $\beta$ of degree $1$ is a linear combination
	of variables $x_{v}$, so it suffices to prove the claim for $\beta=x_{v}$,
	where $v\in\Delta$ is an arbitrary vertex. Since we have already
	established multiplication by $x_{v}$ is equivalent to the composition:
	\[
	A\overset{\mathrm{restriction}}{\rightarrow}A(\mathrm{st}_{v}\Delta)\overset{x_{v}\cdot}{\rightarrow}A(\mathrm{st}_{v}^{\circ}\Delta)\rightarrow A,
	\]
	the claim follows from the assumption that $\alpha$ vanishes in $A(\mathrm{st}_{v}\Delta)$.
\end{proof}
We prove \Cref{property:prop_1} holds for $B$. At the beginning
of \Cref{Poincare} we sketched a computation that $A_{d+1}\simeq\K.$
Let us carry it out in more detail. Schenzel's formula states
\[
\dim_{\K}(\K[\Delta]/\langle \Theta\rangle )_{d+1}=h_{d+1}(\Delta)+\sum_{i=0}^{d-1}(-1)^{i+d+1}\cdot\dim_{\K}H^{i}(\Delta;\K).
\]
The sum on the right is $(-1)^{d+1}(\chi(\Delta)-(-1)^{d}\dim_{\K}H^{d}(\Delta;\K))$,
where $\dim_{\K}H^{d}(\Delta;\K)=1$ since $\Delta$ is
a closed orientable manifold of dimension $d$. Together with the
fact $h_{d+1}(\Delta)=(-1)^{d}\chi(\Delta)$
(as follows directly from the definition $\chi(\Delta)=\sum_{i}(-1)^{i}f_{i}(\Delta)$
and our calculation at the end of \Cref{sec:basic_idea}), this implies
the claim. The same holds for $B$, since $J_{d+1}=0$ by definition.

\Cref{property:prop_2} is a consequene of the Noether isomorphism
theorems: let $v\in\Delta^{(0)}$. Since $B$ is a quotient
of $A$ by an ideal contained in the kernel of the restriction map,
the map factors through $B$. Similarly, \Cref{property:prop_3} follows
directly from the definition: each $\alpha\in B$ of degree at most
$d$ has a nonzero image in $\bigoplus_{v\in\Delta^{(0)}}A(\mathrm{st}_{v}\Delta)$,
and in particular in $A(\mathrm{st}_{v}\Delta)$ for some vertex $v$,
else $\alpha=0$ since its preimage in $A$ is an element of $J$.

Finally, we establish \Cref{property:prop_4} in two steps: first
we show $A(\mathrm{st}_{v}^{\circ}\Delta)$ injects into
$A$, then we show it also injects into $B$.
\begin{prop}
	\label[proposition]{prop:ost_A_inclusion}For $\Delta$ a manifold
	and $v\in\Delta^{(0)}$, the map $A(\mathrm{st}_{v}^{\circ}\Delta)\rightarrow A=A(\Delta)$
	is injective.
\end{prop}
\begin{proof}
	For any simplicial complex $\Sigma$ and vertex $v\in\Sigma$ there
	is a short exact sequence
	\[
	0\rightarrow\K[\mathrm{st}_{v}^{\circ}\Sigma]\rightarrow\K[\Sigma]\rightarrow\K[\Sigma-v]\rightarrow0,
	\]
	where $\Sigma-v=\left\{ \tau\in\Delta\mid v\notin\tau\right\} $ is
	the anti-star of $v$. Applying this to $\Sigma=\mathrm{st}_{\tau}\Delta$
	for each $\tau\in\Delta$ (including $\emptyset$) gives the short
	exact sequence of double complexes
	\[
	0\rightarrow P^\ast(\mathrm{st}_{v}^{\circ}\Delta)\otimes K^\ast(\Theta)\rightarrow P^\ast(\Delta)\otimes K^\ast(\Theta)\rightarrow P^\ast(\Delta-v)\otimes K^\ast(\Theta)\rightarrow0.
	\]
	
	What this means is that these maps of double complexes commute with
	the differentials, and restrict to exact sequences
	\[
	0\rightarrow P^{i}(\mathrm{st}_{v}^{\circ}\Delta)\otimes K^{j}\rightarrow P^{i}(\Delta)\otimes K^{j}\rightarrow P^{i}(\Delta-v)\otimes K^{j}\rightarrow0
	\]
	for each $i,j$. Thus the short exact sequence of double complexes
	restricts to a short exact sequence of rows, for each row. It also
	restricts to a short exact sequence of total complexes, and to a short
	exact sequence of the $(-1)$-th column. We shall use each
	of these in turn.
	
	\subsubsection*{Rows of the double complex sequence}
	
	For each $i$ we have a short exact sequence of the $i$-th rows:
	\[
	0\rightarrow P^\ast(\mathrm{st}_{v}^{\circ}\Delta)\otimes K^{i}(\Theta)\rightarrow P^\ast(\Delta)\otimes K^{i}(\Theta)\rightarrow P^\ast(\Delta-v)\otimes K^{i}(\Theta)\rightarrow0,
	\]
	which yields a long exact sequence in cohomology. Let us understand
	what this sequence really is: it suffices to examine the case $i=0$,
	because the $i$-th row is isomorphic, as a vector space, to the ${d+1 \choose i}$-th
	power of the $0$-th. Up to a change in grading, this is just a short
	exact sequence of the partition complexes:
	\[
	0\rightarrow P^\ast(\mathrm{st}_{v}^{\circ}\Delta)\rightarrow P^\ast(\Delta)\rightarrow P^\ast(\Delta-v)\rightarrow0,
	\]
	which is exact in positive degrees. Thus to understand the homology
	it suffices to restrict it to degree $0$.
	
	Now, $P^\ast(\Delta)_{0}$ is the \v{C}ech complex of $\Delta$
	covered by its stars of vertices, which is isomorphic to the complex
	of simplicial cochains. The same holds for $P^\ast(\Delta-v)_{0}$,
	and the map 
	\[
	P^\ast(\Delta)_{0}\rightarrow P^\ast(\Delta-v)_{0},
	\]
	induced from the restriction map $\K[\Delta]\rightarrow\K[\Delta-v]$,
	is precisely the map of \v{C}ech complexes induced from the simplicial
	map corresponding to the inclusion $\Delta-v\rightarrow\Delta$. This
	in particular means the map
	\[
	H^{i}(P^\ast(\Delta)_{0})\rightarrow H^{i}(P^\ast(\Delta-v)_{0})
	\]
	is the map $H^{i}(\Delta;\K)\rightarrow H^{i}(\Delta-v;\K)$
	in simplicial cohomology induced by the inclusion $\Delta-v\rightarrow\Delta$.
	It is then a topological result that this is a surjection for every
	$i$: this is not difficult to prove using the Mayer-Vietoris sequence
	for cohomology, using the decomposition $\Delta=(\Delta-v)\cup\mathrm{st}_{v}\Delta$
	and the fact $H^{d}(\Delta-v)=0$. Similarly, it is an injection except
	at $i=d$.
	
	Denote $C^{i,j}(\Sigma)=P^i(\Sigma)\otimes K^j(\Theta)$ for each complex $\Sigma$ among $\{\mathrm{st}^\circ_v\Delta, \Delta, \Delta - v\}$. For each row $j$ we obtain an exact sequence:
	\[\ldots \overset{0}{\rightarrow} H^i(C^{\ast,j}(\mathrm{st}^\circ_v\Delta)) \rightarrow H^i(C^{\ast,j}(\Delta)) \twoheadrightarrow H^i(C^{\ast,j}(\Delta - v)) \overset{0}{\rightarrow} H^{i+1}(C^{\ast,j}(\mathrm{st}^\circ_v\Delta))\rightarrow\ldots \]
	
	\subsubsection*{Total complexes and columns of the double complex sequence}
	
	With our understanding of the rows in hand, we examine the total complexes.
	Recall that in the proof of \Cref{thm:total_partition_complex} we
	saw that if $C^{\ast,\ast}(\Sigma)=P^\ast(\Sigma)\otimes K^\ast(\Theta)$,
	$\Sigma$ being any of the three complexes under consideration, then
	there are isomorphisms
	\[
	H^{i}(\mathrm{Tot}(C^{\ast,\ast})_{j}^\ast)\simeq H^{i-s}(C_{j}^{\ast,s}).
	\]
	for each $i,j$ and for $s=d+1-j$. These isomorphisms commute with
	the maps $C^{\ast,\ast}(\mathrm{st}_{v}^{\circ}\Delta)\rightarrow C^{\ast,\ast}(\Delta)$
	and $C^{\ast,\ast}(\Delta)\rightarrow C^{\ast,\ast}(\Delta-v)$ by the naturality
	of the snake lemma. Thus there is a long exact sequence
	\[\ldots \overset{0}{\rightarrow} H^{i}(\mathrm{Tot}(C^{\ast,\ast}(\mathrm{st}^\circ_v\Delta))_{j}^\ast) \rightarrow H^{i}(\mathrm{Tot}(C^{\ast,\ast}(\Delta))_{j}^\ast) \twoheadrightarrow H^{i}(\mathrm{Tot}(C^{\ast,\ast}(\Delta-v))_{j}^\ast) \overset{0}{\rightarrow} \ldots\]
	corresponding to the long exact sequence on the homologies of the rows. This implies the maps
	\[
	H^{i}(\mathrm{Tot}(C^{\ast,\ast}(\Delta))_{j}^\ast)\rightarrow H^{i}(\mathrm{Tot}(C^{\ast,\ast}(\Delta-v))_{j}^\ast)
	\]
	are surjective for all $i$ and injective except when $i+j=2d+1$, because the same occurs for the corresponding maps on the homologies of the rows.
	
	Finally, in the proof of Schenzel's theorem we had isomorphisms 
	\[
	H^{i-1}(\mathrm{Tot}(C^{\ast,\ast})^\ast)_{j}\simeq H^{i}(C^{-1,\ast})_{j},
	\]
	again commuting with $C^{\ast,\ast}(\mathrm{st}_{v}^{\circ}\Delta)\rightarrow C^{\ast,\ast}(\Delta)$
	and with $C^{\ast,\ast}(\Delta)\rightarrow C^{\ast,\ast}(\Delta-v)$. Note this
	means 
	\[
	H^{i}(C^{-1,\ast}(\Delta))_{j}\rightarrow H^{i}(C^{-1,\ast}(\Delta-v))_{j}
	\]
	is surjective for all $i,j$ and injective unless $i+j=2d+2$. 
	Recalling that
	\[
	H^{d+1}(C^{-1,\ast}(\Sigma))_{j}\simeq(\K[\Sigma]/\langle\Theta\rangle)_{j}=A(\Sigma)_{j},
	\]
	we obtain the exact sequence
	\[
	H^{d}(C^{-1,\ast}(\Delta))_{j}\overset{f_{1}}{\rightarrow}H^{d}(C^{-1,\ast}(\Delta-v))_{j}\overset{f_{2}}{\rightarrow}A(\mathrm{st}_{v}^{\circ}\Delta)_{j}\overset{f_{3}}{\rightarrow}A(\Delta)_{j}\rightarrow A(\Delta-v)_{j}\rightarrow0,
	\]
	where $f_{1}$ is surjective for each $j$, so by exactness $f_{2}=0$ and $f_{3}$ is an injection.
\end{proof}
The last piece of the proof of \Cref{property:prop_4} is simply another step similar to those above.
\begin{prop}
	\label[proposition]{prop:ost_B_inclusion}For $\Delta$ a manifold
	and $v\in\Delta^{(0)}$, the map $A(\mathrm{st}_{v}^{\circ}\Delta)\rightarrow B=B(\Delta)$
	is injective.
\end{prop}
\begin{proof}
	Consider the following commutative diagram, in which the ideals $J$
	are those from the definition of $B$.
	\[
	\xymatrix{ &  & 0\ar[d] & 0\ar[d]\\
		& 0\ar[r]\ar[d] & J(\Delta)\ar[d]\ar[r] & J(\Delta-v)\ar[d]\ar[r] & 0\\
		0\ar[r] & A(\mathrm{st}_{v}^{\circ}\Delta)\ar[d]\ar[r] & A(\Delta)\ar[d]\ar[r] & A(\Delta-v)\ar[d]\ar[r] & 0\\
		0\ar[r] & A(\mathrm{st}_{v}^{\circ}\Delta)\ar[r] & B(\Delta)\ar[r] & B(\Delta-v)\ar[r] & 0
	}
	\]
	The columns of the diagram are exact by definition, and the middle
	row 
	\[
	0\rightarrow A(\mathrm{st}_{v}^{\circ}\Delta)_{j}\rightarrow A(\Delta)_{j}\rightarrow A(\Delta-v)_{j}\rightarrow0
	\]
	was shown to be exact at the end of the previous proof. Thus if we
	show the map $J(\Delta)\rightarrow J(\Delta-v)$ is an isomorphism,
	the snake lemma will prove the last row is exact as well, establishing
	the proposition. During our proof of the partition of unity theorem
	we proved
	\[
	H^{-1}(P^\ast/\langle \Theta\rangle )\simeq H^{d}(\mathrm{Tot}(P^\ast\otimes K^\ast)^\ast)
	\]
	for Buchsbaum complexes. In the previous proof, we established 
	\[
	H^{d}(\mathrm{Tot}(P^\ast(\Delta)\otimes K^\ast)^\ast)\rightarrow H^{d}(\mathrm{Tot}(P^\ast(\Delta-v)\otimes K^\ast)^\ast)
	\]
	is surjective in all degrees and injective except in degree $d+1$.
	Since $J$ is defined to be precisely $H^{-1}(P^\ast/\langle \Theta\rangle )$
	in all but degree $d+1$ (in which it is $0$), $J(\Delta)\simeq J(\Delta-v)$
	by the map induced from the restriction, as required.
\end{proof}

\subsection{Further remarks on Poincar\'e duality}

First, let us notice that \Cref{thm:pd} describes a unique way to get a Poincar\'e duality algebra. Indeed, to any element $\gamma$ of a polynomial ring $A$ generated in degree one, one can associate a unique quotient algebra with fundamental class $[\gamma]$. So, the face ring of a closed orientable pseudomanifold of dimension $d-1$ has a unique quotient that is a Poincar\'e algebra with respect to the unique fundamental class in degree $d$.

However, the phenomenon extends further. As proven in \cite{AHL}, \Cref{thm:pd} and its proof extend immediately to prove that

\begin{thm}
For a triangulated orientable manifold $M$ of dimension $d-1$, we have a perfect pairing
\[B_k(M)\ \times\ B_{d-k}(M,\partial M)\ \rightarrow\ B_{d}(M,\partial M)\]
\end{thm} 

\section{Applications: Triangulations and a conjecture of K\"uhnel}

We now will spend two sections providing basic, and mostly inductive, combinatorial applications of partition of unity, although more sophisticated applications are left out here, and relegated to \cite{AHL}.  As the ideas employed are adaptations
of some of those in earlier parts of the paper, we allow ourselves
to use the methods freely at this point. We hope the proof demonstrates
the simplicity and versatility of the tools involved. An important property in this context is the Lefschetz property. 

\begin{defn}
	A Cohen-Macaulay complex $\Delta$ of dimension $d$ has the generic
	strong Lefschetz property if for a generic sequence $\Theta=(\theta_{1},\ldots,\theta_{d+1})$
	in $\K[\Delta]_{1}$ and a generic $\ell\in(\K[\Delta]/\langle \Theta\rangle )_{1}$,
	the map
	\[
	\ell^{d+1-2j}\cdot:(\K[\Delta]/\langle \Theta\rangle )_{j}\rightarrow(\K[\Delta]/\langle \Theta\rangle )_{d+1-j}
	\]
	is injective for each $j\le\frac{d}{2}$.
	
	The generic ``almost Lefschetz property'' is the weaker demand that
	$\ell^{d-2j}\cdot$ be an injection for each $j$, under the same
	conditions.
	
	The Hard Lefschetz theorem is the statement that the strong Lefschetz
	property holds for some $\Theta$.
	
	For $\Delta$ the face ring of a sphere, the Hard Lefschetz theorem
	was proved in \cite{AHL}, with the caveat that the linear system has to be chosen sufficiently generic. Special cases were known earlier, a main
	one being the result for face rings of simplicial polytopes: in this
	case geometric tools are available. See \cite{StanleyHL,McMullenInvent}. 
\end{defn}

\subsection{The inductive principle and partition of unity}

The most basic application of partition of unity is a tool for induction. Indeed, assume that we know the strong Lefschetz property for spheres of dimension $d-1$. Then we can conclude the almost Lefschetz property for closed (not necessarily orientable) manifolds of dimension $d$. Indeed, this is simple: We have
\[\begin{tikzcd}
	B_{j}(M) \arrow{r}{\ \cdot \ell^{d-2j}\ }  \arrow[hook]{d}{}  & B_{d-j}(M) \arrow[hook]{d}{} \\
	\bigoplus_{v\in M^{(0)}}{A}_{j}(\mathrm{st}_v M) \arrow[hook]{r}{\ \cdot \ell^{d-2j}\ } & \bigoplus_{v\in M^{(0)}}{A}_{d-j}(\mathrm{st}_v M)  
	\end{tikzcd}
\]

\subsection{K\"uhnel's efficient triangulations of manifolds}

Poincar\'e duality combines the importance of two maps: The map 
\[A(\Delta)\ \rightarrow\ \bigoplus_{v\in\Delta^{(0)}}A(\mathrm{st}_v\Delta)\]
elevation of the trivial surjection $A(\Delta)\rightarrow A(\mathrm{st}_v\Delta)$,
which is described by the partition theorem, and the map
$A(\mathrm{st}_v^\circ\Delta)\ \rightarrow\ A(\Delta) $
whose properties are described by our analysis of Schenzel's work. We can also elevate this map to the sum
\[\bigoplus_{v\in\Delta^{(0)}} A(\mathrm{st}_v^\circ\Delta)\ \rightarrow\ A(\Delta) \]
and once again, we trivially obtain a surjection in every positive degree.
	
This is quite powerful: K\"uhnel \cite{Kuhnel} famously asked how small a triangulation of a manifold can be chosen. We provide such a bound here, on the number of vertices. For closed, orientable manifolds, this is a result of Murai \cite{MT}, though our proof is simpler.

	Consider now any element $\eta$ of $A_{d-2j+2}(M)$, for instance, the $(d-2j+2)$-th power of a degree one element. Consider the diagram
	\[\begin{tikzcd}
	A_{j}(M) \arrow{r}{\ \cdot \eta\ }  & A_{d-j+2}(M) \\
	\bigoplus_{v\in M^{(0)}}{A}_{j-1}(\mathrm{st}_v M) \arrow[two heads]{u}{} \arrow[two heads]{r}{\ \cdot \eta\ } & \bigoplus_{v\in M^{(0)}}{A}_{d-j+1}(\mathrm{st}_v M) \arrow[two heads]{u}{} 
	\end{tikzcd}
	\]
	where the vertical maps are the cone lemmas, given by the composition
	\[{A}_{j-1}(\mathrm{st}_v M)\ \simeq\ {A}_{j}(\mathrm{st}_v^\circ M)\ \longrightarrow\ {A}_{j}(M)\]
	where the last map is the inclusion of ideals. 
	Now, following the Lefschetz property, the bottom map is a surjection. Hence, the top horizontal map is a surjection. 
	
By the partition of unity theorem, the kernel of the top map is of dimension at least $\binom{d+1}{j}\dim_{\K }H^{j-1}(M;\K )$, and the image is of dimension at least $\binom{d+1}{j-1}\dim_{\K }H^{d-j+1}(M;\K )$.	

It follows that $A_{j}(M)$ is of dimension at least
	\[\binom{d+1}{j}\dim_{\K }H^{j-1}(M;\K )\ +\ \binom{d+1}{j-1}\dim_{\K }H^{d-j+1}(M;\K )\]

In particular, following \Cref{Macaulay}, we have
\[\binom{d+1}{j}\mathrm{b}_{j-1}(M)\ +\ \binom{d+1}{j-1}\mathrm{b}_{d-j+1}(M)
		\ \le \ \binom{n-d+j}{j}\ \quad \text{for}\ 1\le j\le \frac{d+1}{2}.\]

\section{Applications: Subdivisions and the almost-Lefschetz property}\label[section]{sec:subdivisions}

In this section we prove a general subdivision theorem, showing that under relatively mild conditions a subdivision of a Cohen-Macaulay
complex has the ``almost-Lefschetz'' property.
\begin{thm}
	\label[theorem]{thm:subdivisions_almost_lefschetz}Let $\Sigma$ be a Cohen-Macaulay
	(finite) ball complex of dimension $d$ (not necessarily simplicial),
	and let $\Delta$ be any subdivision of $\Sigma$ satisfying the following
	condition: for each face $\sigma\in\Sigma$ of dimension at least $\frac{d}{2}$, the subdivision induced
	on $\partial\sigma$ by $\Delta$ is an induced subcomplex of $\Delta$.
	Then $\Delta$ has the almost-strong Lefschetz property: for a generic
	linear system of parameters $\langle \Theta\rangle $ in
	$\K[\Delta]$, there is a linear form $\omega$ such that
	\[
	\omega^{d-2i-1}:(\K[\Delta]/\langle \Theta\rangle )_{i}\rightarrow(\K[\Delta]/\langle \Theta\rangle )_{d-i-1}
	\]
	is injective.
\end{thm}
This generalizes several results, and applies to a wide range of subdivisions, for instance barycentric subdivisions \cite{KN2} (who proved this only for a very special case of shellable complexes), antiprism subdivisions introduced by Izmestiev and Joswig \cite{IJ}, and interval subdivisions of \cite{AN}; and even for these subdivisions, there was no Lefschetz property known in this generality.

Among the more sophisticated connections, it generalizes the main result of \cite{JS} (by seeing balanced subdivisions of the sphere as derived subdivisions of a cell complex) and, for ball subdivisions of spheres, it is in turn a a special case of the biased pairing theorem \cite[Section 5.4]{AHL}.

The use of the Hard Lefschetz theorem is again inductive, on the links of vertices
in the interiors of (subdivided) faces $\sigma\in\Sigma$.

The next several subsections carry out necessary preparations. These
are, in order:
\begin{enumerate}
	\item Definitions of complexes and subdivisions,
	\item A partition of unity theorem for disks with induced boundary,
	\item Discussion of the Koszul complex and the Lefschetz property for spheres,
	and
	\item Introduction of a modified partition complex.
\end{enumerate}
The proof is easily put together once these pieces are ready. The
idea is that for a facet $K$ of $\Sigma$, the simplicial subdivision
$K^{\prime}\subset\Delta$ has the almost Lefschetz property: this
is a general theorem for simplicial disks with induced boundary. Once
it is known, we use the Cohen-Macaulay property to deduce the global
statement for $\Sigma$.

In other words, we can formulate the key auxiliary result as follows, from which the result follows via inductive principle:

\begin{thm}
With notation as in \Cref{thm:subdivisions_almost_lefschetz}, we have
\[\K[\Delta]/\langle \Theta\rangle\ \longrightarrow\ \bigoplus_{v}\K[\mathrm{st}_v \Delta]/\langle \Theta\rangle\]
is injective for all degrees at most $\frac{d}{2}$, and where $v$ ranges over vertices in the interior of facets of $\Sigma$.
\end{thm}

\subsection{Complexes and subdivisions}
We use the following
definition of a ball complex, which is common in PL topology (we refer to \cite{RourkeSanders-new} for a basic introduction). By polyhedron
we mean a topological space homeomorphic to the realization of a finite
simplicial complex.
\begin{defn}
	A ball complex is a finite set $\Sigma$ of closed disks covering
	a polyhedron such that if $\sigma_{1},\sigma_{2}\in\Sigma$ then:
	\begin{enumerate}
		\item $\sigma_{1}^{\circ}\cap\sigma_{2}^{\circ}=\emptyset$, and
		\item Each of $\partial\sigma_1$ and $\sigma_{1}\cap\sigma_{2}$ is a union
		of disks in $\Sigma$.
	\end{enumerate}
\end{defn}
More general definitions are possible. For instance, one can relax
finiteness to local finiteness.
\begin{defn}
	A simplicial subdivision $\Delta$ of a complex $\Sigma$ is a simplicial
	complex triangulating the polyhedron $\bigcup_{\sigma\in\Sigma}\sigma$
	in such a way that each face of $\Sigma$ is the image of a closed
	subcomplex.
	
	The subdivision \emph{has induced boundaries} if for each $\sigma\in\Sigma$
	the subdivision $\Delta$ induces on $\partial\sigma$ is an induced
	subcomplex of $\Delta$.
\end{defn}

\subsection{Partition of unity for disks with induced boundary}

Let $\Delta$ be a triangulated disk of dimension $d$, and suppose
$\partial\Delta$ is an induced subcomplex. Then a strengthened partition
of unity theorem holds: fewer direct summands are required in the
partition complex. In particular, it suffices to replace $P^{0}$
with a sum over stars of interior vertices only. This is useful for
us because the links of such vertices have the Lefschetz property,
being homology spheres. 

We denote $\Delta^{\circ}=(\Delta,\partial\Delta)$. A \emph{strongly interior face} $\tau\in\Delta$ is one for which
$\tau^{(0)}\subset \Delta^{\circ}$. That is, each vertex of $\tau$ is an interior vertex
(this is stronger than $\tau$ being an interior face, which merely
means it is not entirely contained in $\partial\Delta$). The interior
partition complex $P_{\mathrm{int}}^\ast=P_{\mathrm{int}}^\ast(\Delta)$
is
\[
0\rightarrow\K[\Delta]\rightarrow\bigoplus_{\substack{v\in\Delta^{(0)}, \\ v\in\Delta^{\circ}}}\K[\mathrm{st}_{v}\Delta]\rightarrow\bigoplus_{\substack {e\in\Delta^{(1)}, \\e^{(0)}\subset\Delta^{\circ}}}\K[\mathrm{st}_{e}\Delta]\rightarrow\ldots\rightarrow\bigoplus_{\substack{\tau\in\Delta^{(d)}, \\ \tau^{(0)}\subset\Delta^{\circ}}}\K[\mathrm{st}_{\tau}\Delta]\rightarrow0,
\]
with maps induced by those of the corresponding \v{C}ech complex.
\begin{thm}
	\label[theorem]{thm:interior_partition_of_unity}Let $\Delta$ a triangulated
	disk of dimension $d$ with $\partial\Delta$ induced and $\Theta$
	a generic linear system of parameters in $\K[\Delta]_{1}$. Then $P_{\mathrm{int}}^\ast$
	is exact, and the map
	\[
	\K[\Delta]/\langle \Theta\rangle \rightarrow\bigoplus_{v\in\Delta^{(0)},v\in\Delta^{\circ}}\K[\mathrm{st}_{v}\Delta]/\langle \Theta\rangle 
	\]
	is injective beneath degree $d+1$.
\end{thm}
\begin{proof}
	Once we prove exactness, the theorem is clear using our previous results:
	for instance, the proof of partition of unity applies verbatim (and
	can be simplified: all total complexes are exact because all rows
	are).
	
	Consider a monomial $x^{\alpha}\in\K[\Delta]$ of degree $j$ and
	denote $\rho=\mathrm{supp}(\alpha)$, noting $\dim\rho=j-1$. Then
	the restriction of $x^{\alpha}$ to $\K[\mathrm{st}_{\tau}\Delta${]}
	does not vanish precisely when $\rho\cup\tau\in\Delta$. Let us consider
	the set of all such strongly interior faces $\tau$: it is clearly
	downwards closed, so it forms a subcomplex $\Sigma$ of $\Delta$.
	Since it contains the trivial face it is never the void complex. We
	split the proof into cases.
	\begin{itemize}
		\item If $\rho$ is an interior face, it has at least one interior vertex
		$v$: if all vertices are in $\partial\Delta$ then $\rho\subset\partial\Delta$,
		since $\partial\Delta$ is induced. In this case $\Sigma=\mathrm{st}_{v}\Sigma$:
		each $\sigma\in\Sigma$ satisfies $\sigma\cup\rho\in\Delta$, and
		in particular $\sigma\cup\{v\}\in\Delta$. Thus $\Sigma$ deformation
		retracts onto $v$ and is acyclic.
		\item If $\rho$ is a boundary face, its link contains at least one interior
		vertex (otherwise $\mathrm{st}_{\rho}\Delta=\rho\ast\mathrm{lk}_{\rho}\Delta$
		is entirely contained in the boundary, which has dimension $d-1$,
		but a triangulated disk is pure). The union of the open stars of interior
		vertices of $\mathrm{lk}_{\rho}\Delta$ within $\mathrm{st}_{\rho}\Delta$
		is then the entire $\left|\Delta\right|^{\circ}\cap\left|\mathrm{st}_{\rho}\Delta\right|$.
		Indeed, if if $x\in\left|\Delta\right|^{\circ}$ is an interior point,
		it is in the interior of some face $\tau\in\Delta$ which must have
		an interior vertex (else $\left|\tau\right|$ is contained in the
		boundary by the inducedness assumption). It follows that 
		\[
		\left\{ \mathrm{st}_{\tau}^{\circ}(\mathrm{lk}_{\rho}\Delta)\mid\tau=\left\{ v_{i_{1}},\ldots,v_{i_{k}}\right\} \text{ is a set of interior vertices of }\mathrm{lk}_{\rho}\Delta\right\} 
		\]
		\[
		=\left\{ \mathrm{st}_{\tau}^{\circ}(\mathrm{lk}_{\rho}\Delta)\mid\tau\in\Sigma\right\} 
		\]
		covers $\left|\Delta\right|^{\circ}\cap\left|\mathrm{lk}_{\rho}\Delta\right|$.
		Further, this collection of open stars forms a good cover, since the
		intersection of open stars of strongly interior faces of $\mathrm{lk}_{\rho}\Delta$
		is again the open star of a strongly interior face of $\mathrm{lk}_{\rho}\Delta$,
		and open stars are contractible. Since $\left|\mathrm{lk}_{\rho}\Delta\right|$
		is homotopy equivalent to its interior $\left|\Delta\right|^{\circ}\cap\left|\mathrm{lk}_{\rho}\Delta\right|$,
		the complex $\Sigma$ is contractible.
	\end{itemize}
\end{proof}
\begin{rem}
	To see the homotopy equivalence between $\left|\Delta\right|^{\circ}\cap\left|\mathrm{lk}_{\rho}\Delta\right|$
	and $\left|\mathrm{lk}_{\rho}\Delta\right|$, notice $\left|\mathrm{st}_{\rho}\Delta\right|-\left|\rho\right|\overset{\mathrm{homeo.}}{\simeq}\left|\mathrm{lk}_{\rho}\Delta\right|\times\left|\rho\right|\times(0,1]$.
	This is in turn homotopy equivalent to both $\left|\mathrm{lk}_{\rho}\Delta\right|$
	and $\left|\mathrm{lk}_{\rho}\Delta\right|\times\left|\rho\right|\times(0,1)$.
	The latter can be identified with the open subset $\left|\mathrm{st}_{\rho}^{\circ}\Delta\right|-\left|\rho\right|$
	of $\left|\Delta\right|$, and is thus a manifold with boundary, so
	the collaring theorem applies. Its interior is $(\left|\Delta\right|^{\circ}\cap\left|\mathrm{lk}_{\rho}\Delta\right|)\times\left|\rho\right|\times(0,1)$,
	which is homotopy equivalent to $\left|\Delta\right|^{\circ}\cap\left|\mathrm{lk}_{\rho}\Delta\right|$.
\end{rem}

\subsection{The Koszul complex and the Lefschetz property}

Our goal is to establish the following lemma, a strengthening of \Cref{thm:Koszul}.
\begin{lem}
	Suppose $\Delta$ is a Cohen-Macaulay complex of dimension $k $ and
	$\Theta$ is a generic sequence of $d+1\ge k+1$ elements in $\K[\Delta]_{1}$.
	Consider $K^\ast(\Theta)$ with the grading in which $\alpha\cdot e_{i_{1}}\wedge\ldots\wedge e_{i_{t}}$
	has degree $\deg\alpha$ for any $\alpha\in\K[\Delta]$. If $\K[\Delta]$
	has the generic Lefschetz property, then $H^{i}(K^\ast(\Theta))=0$
	for all $i\le k$ and $H^{i}(K^\ast)_{j}=0$ for all $i>k$
	and $j\le\frac{k}{2}$.
\end{lem}
We will soon be dealing with double complexes in which the columns
are not exact up to the last position, but we still want a partition
of unity theorem to hold, at least in low degrees. The lemma gives
the necessary exactness property. It is simplest to prove with the "naive" grading described in the statement - the corollary contains the statement we will use later.
\begin{proof}
	We return to a remark at the end of \Cref{sec:Koszul}. There is a
	simple relation between $K^\ast(\theta_{1},\ldots,\theta_{j})$ and
	$K^\ast(\theta_{1},\ldots,\theta_{j+1})$: the latter is isomorphic
	to the mapping cone of the chain map
	\[
	K^\ast(\theta_{1},\ldots,\theta_{j})\rightarrow K^\ast(\theta_{1},\ldots,\theta_{j})
	\]
	induced from multiplication by $\theta_{j+1}$. More explicitly, there
	is a short exact sequence of chain complexes of the form
	\[
	0\rightarrow K^{\ast-1}(\theta_{1},\ldots,\theta_{j})\overset{\iota}{\rightarrow}K^\ast(\theta_{1},\ldots,\theta_{j+1})\overset{\pi}{\rightarrow}K^\ast(\theta_{1},\ldots,\theta_{j})\rightarrow0
	\]
	where the maps are given by 
	\[
	\iota(\alpha\cdot e_{i_{1}}\wedge\ldots\wedge e_{i_{t-1}})=\alpha\cdot e_{i_{1}}\wedge\ldots\wedge e_{i_{t-1}}\wedge e_{j+1}
	\]
	and, on basis elements,
	\[
	\pi(e_{i_{1}}\wedge\ldots\wedge e_{i_{t}})=\begin{cases}
	0 & j+1\text{ is among }i_{1},\ldots,i_{t}\\
	e_{i_{1}}\wedge\ldots\wedge e_{i_{t}} & \text{otherwise}.
	\end{cases}
	\]
	It can be verified the connecting homomorphism is $\theta_{j+1}\cdot$
	up to sign. Denoting $K_{j}^\ast=K^\ast(\theta_{1},\ldots,\theta_{j})$
	and similarly $K_{j+1}^\ast$, this means there is a long exact sequence:
	\[
	\ldots\rightarrow H^{i}(K_{j}^\ast)\overset{\pm\theta_{j+1}\cdot}{\rightarrow}H^{i}(K_{j}^\ast)\rightarrow H^{i+1}(K_{j+1}^\ast)\rightarrow H^{i+1}(K_{j}^\ast)\overset{\mp\theta_{j+1}\cdot}{\rightarrow}H^{i+1}(K_{j}^\ast)\rightarrow\ldots
	\]
	In particular, for $j=k+1$, $H^{i}(K_{k+1}^\ast)=0$ unless
	$i=k+1$, in which case the homology is $\K[\Delta]/\langle \theta_{1},\ldots,\theta_{k+1}\rangle $.
	We see $H^{i}(K_{k+2}^\ast)\neq0$ only if $i=k+1$ or $i=k+2$, in
	which case it fits into a short exact sequence in which the other
	nonzero members are either $0$, or the kernel and cokernel of multiplication
	by $\theta_{k+1}$. These are, more explicitly,
	\[
	0\rightarrow H^{k+1}(K_{k+2}^\ast)\rightarrow\ker\left[\K[\Delta]/\langle \theta_{1},\ldots,\theta_{k+1}\rangle \overset{\theta_{k+2}\cdot}{\rightarrow}\K[\Delta]/\langle \theta_{1},\ldots,\theta_{k+1}\rangle \right]\rightarrow0
	\]
	and
	\[
	0\rightarrow\operatorname{coker}\left[\K[\Delta]/\langle \theta_{1},\ldots,\theta_{k+1}\rangle \overset{\theta_{k+2}\cdot}{\rightarrow}\K[\Delta]/\langle \theta_{1},\ldots,\theta_{k+1}\rangle \right]\rightarrow H^{k+1}(K_{k+2}^\ast)\rightarrow0.
	\]
	For $j>k+1$, by induction, each $H^{i}(K_{j+1}^\ast)$
	fits into a short exact sequence of $\K [\Delta]$-modules
	which vanish in each degree $\le\frac{k}{2}$, and is $0$ unless
	$i\ge k+1$.
\end{proof}
\begin{cor}\label[corollary]{cor:subd_Koszul}
	Under the conditions of the previous lemma, consider $K^\ast(\Theta)$
	with the grading in which $\deg(\alpha\cdot e_{i_{1}}\wedge\ldots\wedge e_{i_{t}})=\deg(\alpha)+d+1-t$.
	Then $H^{i}(K^\ast)_{j}=0$ for all $i\le k+1$ and $j\le\frac{d}{2}$.\label[corollary]{cor:subd_koszul}
\end{cor}
\begin{proof}
	Consider a cycle $z$ in $H^{i}(K^\ast)_{j}$: it is a sum
	of elements of the form $\alpha e_{t_{1}}\wedge\ldots\wedge e_{t_{i}}$,
	with $\deg(\alpha)=i+j-d-1$. We have $\deg(\alpha)>\frac{k}{2}$
	by the previous theorem, so
	\[
	i+j>d+1+\frac{k}{2}.
	\]
	If $i<k+1$ we already know $H^{i}(K^\ast)=0$. If $i=k+1$,
	this implies $j>d-\frac{k}{2}\ge\frac{d}{2}$. 
\end{proof}

\subsection{A modified partition complex}

Consider complex $\Sigma$ as in the theorem. It has a cellular chain
complex, which computes its homology. We will work with a similar
chain complex, with ``augmentation'' at the wrong end. For $\sigma\in\Sigma$,
let us denote by $\K[\sigma]$ the face ring of the subdivision induced
on $\sigma$ by $\Delta$. The complex is:
\[
\tilde{P}^\ast(\Sigma)=0\rightarrow\K[\Delta]\rightarrow\bigoplus_{\tau\in\Sigma^{(d)}}\K[\tau]\rightarrow\bigoplus_{\rho\in\Sigma^{(d-1)}}\K[\rho]\rightarrow\ldots\rightarrow\bigoplus_{v\in\Sigma^{(0)}}\K\rightarrow0.
\]
Note that without the $\K[\Delta]$ at the beginning, the $0$-th
graded piece computes $H_{\ast}(\Sigma;\K)$. Nevertheless we keep $\tilde{P}^\ast$
cohomologically graded in order to make preceding arguments easier
to use in this setting. The main result on this complex is the following.
\begin{thm}
	Let $\Theta$ be a generic sequence of $d+1$ elements in $\K[\Delta]_{1}$.
	The kernel $H^{-1}(\tilde{P}^\ast/\langle \Theta\rangle )$
	of $\K[\Delta]/\langle \Theta\rangle \rightarrow\bigoplus_{\tau\in\Sigma^{(d)}}\K[\tau]/\langle \Theta\rangle $
	vanishes in each degree $i\le\frac{d}{2}$.
\end{thm}
The idea is similar to the proof of the partition of unity theorem.
We begin by discussing the needed modifications.

\subsubsection{The unreduced partition complex}

Consider a single monomial $x^{\alpha}$, and the summands of $\tilde{P}^\ast$
in which $x^{\alpha}$ is nonzero: these are the faces containing
$\rho=\mathrm{supp}(\alpha)$, together with $\K[\Delta]=\tilde{P}^{-1}$.
As a poset with the inclusion relation, the collection of faces corresponding
to these summands (other than to $\K[\Delta]$ itself) is homotopy
equivalent to $\mathrm{lk}_{\rho}\Delta$, and has no homology beneath
the top dimension $s=d-|\rho|$. Observe that a face of dimension
$s$ in $\mathrm{lk}_{\rho}\Delta$ corresponds to the summand of
some $\sigma\in\Sigma^{(s+|\rho|)}$, and that if $\rho\neq\emptyset$,
it is the reduced homology of $\mathrm{lk}_{\rho}\Delta$ which is
computed by these summands since $\rho$ itself is contained in some
unique minimal face of $\Sigma$. Thus $H^{i}(\tilde{P}^\ast)=0$
unless $i=0$ or $i=d$.

\subsubsection{Homology of the total complex}

Set $C^{\ast,\ast}=\tilde{P}^\ast\otimes K^\ast(\Theta)$. We follow
the strategy of the partition of unity theorem, and ideas from that
proof are used freely here. The grading of $C^{\ast,\ast}$ is also from
that part of the paper; it is the same grading recalled in \Cref{cor:subd_Koszul}. 

The main differences between what follows and the proof of partition
of unity are that we work only with the part of $C^{\ast,\ast}$ in degree
$\le\frac{d}{2}$, and that it is possible that $H^{0}(\tilde{P}^\ast)\neq0$. 

Two observations are necessary at this point. The first is clear:
in low degrees, all columns of the augmented double complex (with
an additional top row $\tilde{P}^\ast/\langle \Theta\rangle $)
are exact. The second is the next lemma.
\begin{lem}
	For each $i\le\frac{d}{2}$, $H^{d}(\mathrm{Tot}(C^{\ast,\ast})^\ast)_{i}=0$.
\end{lem}
\begin{proof}
	Let $i\le\frac{d}{2}$. Consider the short exact sequence associated
	with the mapping cone of $\mathfrak{U}^{d}:\mathrm{Tot}(C^{\ast,\ast\le d})^\ast\rightarrow C^{\ast-d,d+1}$:
	it is
	\[
	\ldots\rightarrow H^{j}(C^{\ast-d-1,d+1})\rightarrow H^{j}(\mathrm{Tot}(C^{\ast,\ast})^\ast)\rightarrow H^{j}(\mathrm{Tot}(C^{\ast,\ast\le d})^\ast)\overset{\partial}{\rightarrow}H^{j+1}(C^{\ast-d,d+1})\rightarrow\ldots,
	\]
	where the connecting homomorphism $\partial$ is induced by the map
	\[
	\alpha=\sum_{\substack{r+s=j,\\
			s\le d
		}
	}\alpha^{r,s}\mapsto d^{v}(\alpha^{j-d,d}).
	\]
	Suppose $\alpha\in H^{j}(\mathrm{Tot}(C^{\ast,\ast\le d})^\ast)_{i}$
	has $\partial\alpha=0$. Then $\alpha$ is in the image of $H^{j}(\mathrm{Tot}(C^{\ast,\ast})^\ast)_{i}$
	in an obvious way: since $d^{v}(\alpha^{j-d,d})=0$, the
	sum 
	\[
	\sum_{\substack{r+s=j,\\
			s\le d
		}
	}\alpha^{r,s}
	\]
	is already a cycle of $\mathrm{Tot}(C^{\ast,\ast})_{i}^\ast$,
	so it represents an element of $H^{j}(\mathrm{Tot}(C^{\ast,\ast})^\ast)_{i}$.
	However, this $\alpha$ has no summand in the top row, and all columns
	of $C^{\ast,\ast}$ are exact beneath the top row in degree $i$. From \Cref{lem:exactness}
	it follows $\alpha$ is zero in $H^{j}(\mathrm{Tot}(C^{\ast,\ast})^\ast)_{i}$,
	hence its image in $H^{j}(\mathrm{Tot}(C^{\ast,\ast\le d})^\ast)_{i}$
	is zero also.
	
	We can conclude $\partial$ is injective in degrees $i\le\frac{d}{2}$.
	Take a piece around $H^{d}(\mathrm{Tot}(C^{\ast,\ast})^\ast)_{i}$
	of the exact sequence from above:
	\[
	\ldots\rightarrow H^{d}(C^{\ast-d-1,d+1})_{i}\overset{f_{1}}{\rightarrow}H^{d}(\mathrm{Tot}(C^{\ast,\ast})^\ast)_{i}\overset{f_{2}}{\rightarrow}H^{d}(\mathrm{Tot}(C^{\ast,\ast\le d})^\ast)_{i}\overset{\partial}{\rightarrow}H^{d+1}(C^{\ast-d,d+1})_{i}\rightarrow\ldots
	\]
	Since $\partial$ is injective, $f_{2}=0$ and $f_{1}$ is surjective.
	Thus $H^{d}(C^{\ast-d-1,d+1})_{i}=H^{-1}(C^{\ast,d+1})_{i}=0$
	surjects onto $H^{d}(\mathrm{Tot}(C^{\ast,\ast})^\ast)_{i}$.
\end{proof}
Together with these observations, directly following the proof of
the partition of unity theorem gives the result.

\subsection{Proof of the subdivision theorem}
\begin{lem}
	For each $\sigma\in\Sigma^{(d)}$, $\K[\sigma]$ has the
	generic almost-Lefschetz property.
\end{lem}
\begin{proof}
	Pick a generic system of parameters $\Theta=(\theta_{1},\ldots,\theta_{d+1})$
	for $\sigma$. The map 
	\[
	\K[\sigma]/\langle \Theta\rangle \rightarrow\bigoplus_{v\in\sigma^{(0)},v\in\sigma^{\circ}}\K[\mathrm{st}_{v}\sigma]/\langle \Theta\rangle 
	\]
	is injective in each degree $\le d$ by \Cref{thm:interior_partition_of_unity}
	(injectivity up to degree $\frac{d}{2}$ suffices). For each interior
	vertex $v$, the cone lemma gives
	\[
	\K[\mathrm{st}_{v}\sigma]/\langle \Theta\rangle \simeq\K[\mathrm{lk}_{v}\sigma]/\langle \Theta^{\prime}\rangle ,
	\]
	a sphere of dimension $d-1$ with the generic Lefschetz property,
	($\Theta^{\prime}$ is a system of $d$ parameters depending on $\Theta$:
	see the cone lemma). 
	
	Choosing a generic $\ell\in\K[\sigma]/\langle \Theta\rangle $,
	the map $\ell^{d-2j}:(\K[\mathrm{st}_{v}\sigma]/\langle \Theta\rangle )_{j}\rightarrow(\K[\mathrm{st}_{v}\sigma]/\langle \Theta\rangle )_{d-j}$
	is therefore injective for each $j\le\frac{d}{2}$. For such $j$,
	the injectivity of the bottom horizontal map in the commutative diagram
	\[
	\xymatrix{(\K[\sigma]/\langle \Theta\rangle )_{j}\ar[r]^{\ell^{d-2j}\cdot}\ar[d] & (\K[\sigma]/\langle \Theta\rangle )_{d-2j}\ar[d]\\
		{\displaystyle \bigoplus_{v\in\sigma^{(0)},v\in\sigma^{\circ}}}(\K[\mathrm{st}_{v}\sigma]/\langle \Theta\rangle )_{j}\ar[r]^{\ell^{d-2j}\cdot} & {\displaystyle \bigoplus_{v\in\sigma^{(0)},v\in\sigma^{\circ}}}(\K[\mathrm{st}_{v}\sigma]/\langle \Theta\rangle )_{d-2j}
	}
	\]
	implies multiplication $\ell^{d-2j}$ is injective on $(\K[\sigma]/\langle \Theta\rangle )_{j}$,
	as required.
\end{proof}
\begin{proof}
	[Proof of the subdivision theorem] This is completely analogous to
	the previous lemma. The map $\K[\Delta]/\langle \Theta\rangle \rightarrow\bigoplus_{\tau\in\Sigma^{(d)}}\K[\tau]/\langle \Theta\rangle $
	is injective up to degree $\frac{d}{2}$. Each summand $\K[\tau]$
	has the generic almost Lefschetz property, and for generic $\ell\in\K[\sigma]/\langle \Theta\rangle $
	the bottom horizontal map in the commutative diagram
	\[
	\xymatrix{(\K[\Delta]/\langle \Theta\rangle )_{j}\ar[r]^{\ell^{d-2j}\cdot}\ar[d] & (\K[\Delta]/\langle \Theta\rangle )_{d-2j}\ar[d]\\
		{\displaystyle \bigoplus_{\tau\in\Sigma^{(d)}}}(\K[\tau]/\langle \Theta\rangle )_{j}\ar[r]^{\ell^{d-2j}\cdot} & {\displaystyle \bigoplus_{\tau\in\Sigma^{(d)}}}(\K[\tau]/\langle \Theta\rangle )_{d-2j}
	}
	\]
	yields the result.
\end{proof}

	{\small
		\bibliographystyle{myamsalpha}
		\bibliography{ref}}

\end{document}